\documentclass[11pt]{article}

\usepackage{amssymb,amsmath,amsfonts}
\usepackage{graphicx,color,enumitem}
\usepackage{amsthm} 
\usepackage{bm}
\usepackage[round]{natbib}

\usepackage{geometry}

\usepackage[colorinlistoftodos, textwidth=4cm, shadow]{todonotes}
\RequirePackage[colorlinks,citecolor=blue,urlcolor=blue]{hyperref}

\newcommand{\E}{{\mathbb E}}
\newcommand{\F}{{\mathbb F}}

\renewcommand{\P}{{\mathbb P}}

\newcommand{\R}{{\mathbb R}}
\renewcommand{\S}{{\mathbb S}}
\newcommand{\N}{{\mathbb N}}

\newcommand{\Ccal}{{\mathcal C}}

\newcommand{\Fcal}{{\mathcal F}}

\newcommand{\Pcal}{{\mathcal P}}

\newcommand{\Xcal}{{\mathcal X}}

\newcommand{\fdot}{{\,\cdot\,}}

\DeclareMathOperator{\tr}{tr}

\newtheorem{theorem}{Theorem}

\newtheorem{corollary}[theorem]{Corollary}
\newtheorem{definition}[theorem]{Definition}

\newtheorem{lemma}[theorem]{Lemma}

\newtheorem{remark}[theorem]{Remark}

\theoremstyle{definition}
\newtheorem{example}[theorem]{Example}

\numberwithin{equation}{section}
\numberwithin{theorem}{section}

\definecolor{darkgreen}{rgb}{0,0.7,0}

\DeclareMathOperator{\diag}{diag}

\newcommand{\iii}{{\vert\kern-0.25ex\vert\kern-0.25ex\vert}}

\begin{document}

\title{A weak solution theory for stochastic Volterra equations of convolution type\footnote{The work of Eduardo Abi Jaber was supported by grants from R\'egion Ile-de-France. Christa Cuchiero gratefully acknowledges financial support by the Vienna Science and Technology Fund (WWTF) under grant MA16-021 and the Austrian Science Fund (FWF) under grant  Y1235 of the START-program. The research of Sergio Pulido benefited from the support of the Chair Markets in Transition (F\'ed\'eration Bancaire Fran\c caise) and the project ANR 11-LABX-0019. Sergio Pulido acknowledges support by the Europlace Institute of Finance (EIF) and the Labex Louis Bachelier, research project: ``The impact of information on financial markets".}}
\author{Eduardo Abi Jaber\thanks{Universit\'e Paris 1 Panth\'eon-Sorbonne, eduardo.abi-jaber@univ-paris1.fr} \and Christa Cuchiero\thanks{University of Vienna, christa.cuchiero@univie.ac.at} \and Martin Larsson\thanks{Carnegie Mellon University, martinl@andrew.cmu.edu} \and Sergio Pulido\thanks{ENSIIE \& Universit\'e Paris-Saclay, sergio.pulidonino@ensiie.fr}}

\maketitle

\begin{abstract}
We obtain general weak existence and stability results for stochastic convolution equations with jumps under mild regularity assumptions, allowing for non-Lipschitz coefficients and singular kernels. Our approach relies on weak convergence in $L^p$ spaces. The main tools are new a priori estimates on Sobolev--Slobodeckij norms of the solution, as well as a novel martingale problem that is equivalent to the original equation. This leads to generic approximation and stability theorems in the spirit of classical martingale problem theory. We also prove uniqueness and path regularity of solutions under additional hypotheses. To illustrate the applicability of our results, we consider scaling limits of nonlinear Hawkes processes and approximations of stochastic Volterra processes by Markovian semimartingales.
\end{abstract}

\tableofcontents

\section{Introduction and main results}\label{S_main}

A stochastic Volterra equation of convolution type is a stochastic equation of the form
\begin{equation}\label{eq_SVEC}
X_t = g_0(t) + \int_{[0,t)} K(t-s)dZ_s,
\end{equation}
where $X$ is the $d$-dimensional process to be solved for, $g_0$ is a given function, $K$ is a given $d\times k$ matrix-valued convolution kernel, and $Z$ is a $k$-dimensional It\^o semimartingale whose differential characteristics are given functions of $X$. The solution concept is described in detail below. In particular, conditions are needed to ensure that the stochastic integral on the right-hand side of \eqref{eq_SVEC} is well-defined.

This type of equation appears in multiple applications, for example turbulence \citep{barndorff2008time}, energy markets \citep{barndorff2013modelling}, and rough volatility modeling in finance \citep{EER:06, volatilityrough2014}. In the latter context the kernel is singular, $K(t)=t^{\gamma-1}$ with $\gamma \in (\frac{1}{2},1)$, and the driving semimartingale is continuous with coefficients that are just continuous functions without any Lipschitz-type regularity. Such examples fall outside the scope of classical theory, such as the results of \cite{BM:80:1, P:85, cou_dec_01, W:08, Z:10}. This motivated the work of \citet{AJLP17}, although their results only apply in the path-continuous case. Equations like \eqref{eq_SVEC} also occur in the study of fractional Brownian motion.

There are however many important examples with jumps. The most basic ones  are L\'evy driven moving averages where the characteristics of the driving semimartingale are constant and thus do not depend on $X$ \citep{BP:09,M:06}. A more complex example is the intensity $\lambda$ of a Hawkes process $N$. Here the driving semimartingale is the Hawkes process itself, which is a counting process, and the intensity satisfies
	\[
	\lambda_t= g_0(t) + \int_{[0,t)} K(t-s)dN_s.
	\]
Various multivariate and nonlinear generalizations have also been studied and applied; see \citet{bremaud1996stability,daley2003introduction,delattre2016hawkes} and the references there.

Solutions of \eqref{eq_SVEC} are neither semimartingales nor Markov processes in general. Classically, they are constructed using Picard iteration, but only under Lipschitz or near-Lipschitz assumptions. Alternatively, one can use scaling limits of Hawkes-type processes to generate continuous solutions for well-chosen kernels and affine characteristics \citep{JR:16, GKR:19}. Yet another approach is to use projections of Markovian solutions to certain degenerate stochastic partial differential equations \citep{AJEE:19b,BDK:19,CT:18,CT:19, MS:15}. In the case of affine characteristics a unified theory is presented by \citet{CT:18}, by lifting Volterra processes to so-called generalized Feller processes in infinite dimension.  Their construction builds  on approximating Brownian or complicated jump drivers by finite activity jump processes.

In this paper we also use approximation by jumps, but not via scaling limits of Hawkes processes, nor infinite dimensional lifts. Instead we work with a priori $L^p$ estimates for solutions of \eqref{eq_SVEC}, combined with a novel ``Volterra'' martingale problem in $\R^d$ that allows us to pass to weak limits in \eqref{eq_SVEC}. In view of the irregular path behavior that occurs, in particular, in the presence of jumps, this identifies $L^p$ spaces as a natural environment for the weak convergence analysis.
With this approach we obtain 
\begin{itemize}
\itemsep0em
\item existence of weak solutions for singular kernels, non-Lipschitz coefficients and general jump behavior;
\item strong existence and pathwise uniqueness under suitable Lipschitz conditions (but still singular kernels and jumps);
\item convergence and stability theorems in the spirit of classical martingale problem theory, allowing for instance to study scaling limits of nonlinear Hawkes processes and to approximate stochastic Volterra processes by Markovian semimartingales;
\item path regularity under certain additional conditions on the kernel and the characteristics.
\end{itemize}

Let us now describe the solution concept for \eqref{eq_SVEC}. For $p\in[2,\infty)$  we denote by $L^p_{\rm loc}=L^p_{\rm loc}(\R_+,\R^n)$ the space of locally $p$-integrable functions from $\R_+$ to $\R^n$, where the dimension $n$ of the image space will depend on the context. Let $d,k\in\N$ and consider the following data:
\begin{enumerate}[label={\rm(D\arabic*)}]
\itemsep0em
\item\label{SVE_IC} an {\em initial condition} $g_0\colon\R_+\to\R^d$ in $L^p_{\rm loc}$,
\item\label{SVE_K} a {\em convolution kernel} $K\colon\R_+\to\R^{d\times k}$ in $L^p_{\rm loc}$,
\item\label{SVE_abnu} a {\em characteristic triplet} $(b,a,\nu)$ of measurable maps $b\colon\R^d\to\R^k$ and $a\colon\R^d\to\S^k_+$ as well as a kernel $\nu(x,d\zeta)$ from $\R^d$ into $\R^k$ such that $\nu(x,\{0\})=0$ for all $x\in\R^d$ and, for some $c\in\R_+$,
\end{enumerate}
\begin{equation}\label{eq_GB_1}
|b(x)| + |a(x)| + \int_{\R^k} {\left(1\wedge|\zeta|^2\right)} \nu(x,d\zeta) \le c (1+|x|^p), \quad x\in\R^d.
\end{equation}
Given this data, we can now state the following key definition.

\begin{definition}\label{D:solSVE}
A {\em weak $L^p$ solution of \eqref{eq_SVEC}} for the data $(g_0,K,b,a,\nu)$ is an $\R^d$-valued predictable process $X$, defined on some filtered probability space $(\Omega,\Fcal,\F,\P)$, that has trajectories in $L^p_{\rm loc}$ and satisfies
\begin{equation}\label{eq_SVEC_def}
X_t = g_0(t) + \int_{[0,t)} K(t-s)dZ_s \quad \text{$\P\otimes dt$-a.e.}
\end{equation}
for some $\R^k$-valued It\^o semimartingale $Z$ with $Z_0=0$ whose differential characteristics (with respect to some given truncation function) are $b(X)$, $a(X)$, $\nu(X,d\zeta)$. For convenience we often refer to the pair $(X,Z)$ as a weak $L^p$ solution.
\end{definition}

{Due to condition \eqref{eq_GB_1},} the stochastic integral in \eqref{eq_SVEC_def} is well-defined for almost every $t\in\R_+$, confirming that the definition of $L^p$ solution makes sense. This is shown in Lemma~\ref{L_SI_welldef}. 

Throughout this section we assume $\int_{\mathbb{R}^k} |\zeta|^2 \nu(x,d\zeta) < \infty$ for all $x \in \mathbb{R}^d$ so we can use the ``truncation function'' $\chi(\zeta)=\zeta$. The characteristics of $Z$ are therefore understood with respect to this function. We can now state our main result on existence of weak $L^p$ solutions.

\begin{theorem}\label{T_weak_existence}
Let $d,k\in\N$, $p\in[2,\infty)$, and consider data $(g_0,K,b,a,\nu)$ as in \ref{SVE_IC}--\ref{SVE_abnu}. Assume $b$ and $a$ are continuous, and $x\mapsto|\zeta|^2 \nu(x,d\zeta)$ is continuous from $\R^d$ into the finite positive measures on $\R^k$ with the topology of weak convergence. In addition, assume there exist a constant $\eta\in(0,1)$, a locally bounded function $c_K\colon\R_+\to\R_+$, and a constant $c_{\rm LG}$ such that
\begin{equation}\label{eq_K_Wpnu}
\int_0^T \frac{|K(t)|^p}{t^{\eta p}}dt + \int_0^T \int_0^T \frac{|K(t)-K(s)|^p}{|t-s|^{1+\eta p}}ds\,dt \le c_K(T), \quad T\ge0,
\end{equation}
and
\begin{equation}\label{eq_LG}
\begin{aligned}
|b(x)|^2 &+ |a(x)| + \int_{\R^k} |\zeta|^2 \nu(x,d\zeta) \\
&\quad + \left(\int_{\R^k} |\zeta|^p \nu(x,d\zeta)\right)^{2/p} \le c_{\rm LG} (1+|x|^2), \quad x\in\R^d.
\end{aligned}
\end{equation}
Then there is a weak $L^p$ solution $(X,Z)$ of \eqref{eq_SVEC} for the data $(g_0,K,b,a,\nu)$.
\end{theorem}

An overview of the proof of Theorem~\ref{T_weak_existence} is given below, and the formal argument is in Section~\ref{S_pf_Tex1}. However, let us first mention several kernels of interest that satisfy \eqref{eq_K_Wpnu}.

\begin{example}\label{ex:kernels}
\begin{enumerate}
\item\label{ex:kernels_1}
Consider the  kernel $K(t)=t^{\gamma-1}$ with $\gamma>\frac12$, which is singular when $\gamma < 1$. Then with $\eta\in(0,(\gamma-\frac12) \wedge 1)$ one has $2\gamma-2\eta-1>0$, and therefore
\[
\int_0^T |K(t)|^2 t^{-2\eta}dt = \frac{T^{2\gamma-2\eta-1}}{2\gamma-2\eta-1}
\]
and  
\[
\int_0^T \int_0^T \frac{|K(t)-K(s)|^2}{|t-s|^{1+2\eta}}ds\,dt = \frac{2T^{2\gamma-2\eta-1}}{2\gamma-2\eta-1} \int_0^1 \frac{(u^{\gamma-1}-1)^2}{(1-u)^{1+2\eta}} du.
\]
These expressions are locally bounded in $T$, so \eqref{eq_K_Wpnu} holds with $p=2$.
\item\label{ex:kernels:2}
Consider a locally Lipschitz kernel $K$ with optimal Lipschitz constant $L_T$  over $[0, T]$. Let $p \in [2,\infty)$ and choose $\eta < \frac{1}{p}$. 
Then
\begin{align*}
\int_0^T |K(t)|^p t^{-\eta p} dt \leq \max_{t\in[0,T]}| K(t)|^p \frac{T^{1-\eta p}}{1-\eta p}
\end{align*} 
and
\[
\int_0^T \int_0^T \frac{|K(t)-K(s)|^p}{|t-s|^{1+2\eta}}ds\,dt \leq 
L_T^p \int_0^T \int_0^T |t-s|^{p-1-2\eta} ds \, dt.
\]
Since $1-\eta p>0$ and hence $p-2\eta > 0$, these expressions are locally bounded in $T$. Thus \eqref{eq_K_Wpnu} holds.
\item\label{ex_K_iii} Consider two kernels $K_1$ and $K_2$. Suppose $K_1 \in L^p_{\rm loc}$ satisfies \eqref{eq_K_Wpnu} for some $p \in [2, \infty)$ and $\eta \in (0,1)$, and $K_2$ is locally Lipschitz. Then it is not hard to check that the product $K=K_1K_2$ satisfies \eqref{eq_K_Wpnu} with the same $p$ and $\eta$ as $K_1$. An example of this kind is the exponentially dampened singular kernel  $K(t)=t^{\gamma-1} e^{-\beta t}$ with $\gamma \in (\frac{1}{2},1)$ and $\beta \geq 0$. For this kernel one can take $p=2$ and any $\eta \in (0,\gamma -\frac{1}{2})$.
\end{enumerate}
\end{example}

The proof of Theorem~\ref{T_weak_existence} is based on approximation and weak convergence of laws on suitable function spaces. The semimartingale $Z$ has trajectories in the Skorokhod space $D=D(\R_+,\R^k)$ of c\`adl\`ag functions. Weak convergence in $D$ is a classical tool used, for example, to obtain weak solutions of stochastic differential equations with jumps (see, e.g., \citet{EK:05}). However, as explained in Section~\ref{s:pathreg}, the trajectories of $X$ need not be c\`adl\`ag, only locally $p$-integrable. Thus it is natural to regard $X$ as a random element of the Polish space $L^p_{\rm loc}=L^p_{\rm loc}(\R_+,\R^d)$. It is in this space---or rather, the product space $L^p_{\rm loc}\times D$---that our weak convergence analysis takes place.

{Relative compactness} in $L^p$ is characterized by the Kolmogorov--Riesz--Fr\'echet theorem; see e.g.\ \citet[Theorem~4.26]{brezis2010functional}. A more convenient criterion in our context uses the Sobolev--Slobodeckij norms, defined for any measurable function $f\colon\R_+\to\R^d$ by
\[
\| f \|_{W^{\eta,p}(0,T)} = \left( \int_0^T |f(t)|^p dt + \int_0^T \int_0^T \frac{|f(t)-f(s)|^p}{|t-s|^{1+\eta p}} ds\, dt \right)^{1/p},
\]
where $p\ge1$, $\eta\in(0,1)$, $T\ge0$ are parameters. The relation between these norms and $L^p$ spaces is somewhat analogous to the relation between H\"older norms and spaces of continuous functions. In particular, balls with respect to $\| \fdot \|_{W^{\eta,p}(0,T)}$ are {relatively compact} in $L^p(0,T)$; see e.g.~\citet[Theorem 2.1]{flandoli1995martingale}. The following a priori estimate clarifies the role of the conditions \eqref{eq_K_Wpnu} and \eqref{eq_LG} in Theorem~\ref{T_weak_existence}, and is the key tool that allows us to obtain convergent sequences of approximate $L^p$ solutions. The proof is given in Section~\ref{S_apriori}.

\begin{theorem}\label{T_apriori}
Let $d,k\in\N$, $p\in[2,\infty)$, and consider data $(g_0,K,b,a,\nu)$ as in \ref{SVE_IC}--\ref{SVE_abnu}. Assume there exists a constant $c_{\rm LG}$ such that  \eqref{eq_LG} holds. Then any weak $L^p$ solution $X$ of \eqref{eq_SVEC} for the data $(g_0,K,b,a,\nu)$ satisfies
\begin{align}\label{eq:estimateLp}
\E[ \|X\|_{L^{p}(0,T)}^p] \le c,
\end{align}
where $c<\infty$ only depends on $d,k,p, c_{\rm LG}, T, \|g_0\|_{L^p(0,T)}$, and, $L^p$-continuously, on $K|_{[0,T]}$. If in addition there exist a constant $\eta\in(0,1)$ and a locally bounded function $c_K\colon\R_+\to\R_+$ such that \eqref{eq_K_Wpnu} holds, then 
\begin{equation}\label{eq:estimateW}
\E[ \|X-g_0\|_{W^{\eta,p}(0,T)}^p] \le c,
\end{equation}
where $c<\infty$ only depends on $d,k,p,\eta, c_K, c_{\rm LG},T$.
\end{theorem}

An immediate corollary is the following tightness result.

\begin{corollary}\label{C_tightness}
Fix $d,k,p,\eta, c_K, c_{\rm LG}$ as in Theorem~\ref{T_apriori}, and let $G_0\subset L^p_{\rm loc}$ be relatively compact. Let $\Xcal$ be the set of all weak $L^p$ solutions $X$ of \eqref{eq_SVEC} as $g_0$ ranges through $G_0$, $K$ ranges through all kernels that satisfy \eqref{eq_K_Wpnu} with the given $\eta$ and $c_K$, and $(b,a,\nu)$ ranges through all characteristic triplets that satisfy \eqref{eq_LG} with the given $c_{\rm LG}$. Then $\Xcal$ is tight, in the sense that the family $\{\text{Law}(X)\colon X\in \Xcal\}$ is tight in $\Pcal(L^p_{\rm loc})$.
\end{corollary}

\begin{proof}
Fix $T\in\R_+$ and let $c$ be the constant in \eqref{eq:estimateW}. For any $m>0$, Markov's inequality gives
\[
\sup_{X\in\Xcal} \P( \|X-g_0\|_{W^{\eta,p}(0,T)} > m) \le \frac{c}{m^p}.
\]
The balls $\{f\colon \|f\|_{W^{\eta,p}(0,T)} \le m\}$ are relatively compact in $L^p(0,T)$, so the above estimate implies that the family $\{(X-g_0)|_{[0,T]}\colon X\in\Xcal\}$ is tight in $L^p(0,T)$. Since $T$ was arbitrary, it follows that $\Xcal_0=\{X-g_0\colon X\in \Xcal\}$ is tight in $L^p_{\rm loc}$. Since $G_0$ is relatively compact, $G_0+\Xcal_0$ is tight as well, and it contains $\Xcal$. Thus $\Xcal$ is tight.
\end{proof}

The second main ingredient in the proof of Theorem~\ref{T_weak_existence} relies on a reformulation of \eqref{eq_SVEC} as a certain martingale problem.  This martingale problem is introduced in Section~\ref{S_stability}, and it is shown in Lemma~\ref{L_SVE_iff_MGP} that weak $L^p$ solutions of \eqref{eq_SVEC} can equivalently be understood as solutions of the martingale problem. This point of view is useful because it leads to the following stability result, which under appropriate conditions asserts that the weak limit of a sequence of solutions is again a solution. The proof is given at the end of Section~\ref{S_stability}. Recall that $D$ denotes the Skorokhod space of c\`adl\`ag functions from $\R_+$ to $\R^k$.

\begin{theorem}\label{T_stability}
Let $d,k\in\N$, $p\in[2,\infty)$. For each $n\in\N$, let $(X^n,Z^n)$ be a weak $L^p$ solution of \eqref{eq_SVEC} given data $(g_0^n,K^n,b^n,a^n,\nu^n)$ as in \ref{SVE_IC}--\ref{SVE_abnu}. Assume the triplets $(b^n,a^n,\nu^n)$ all satisfy \eqref{eq_LG} with a common constant $c_{\rm LG}$.  Assume also, for some $(g_0,K,b,a,\nu)$ and limiting process $(X,Z)$, that:
\begin{itemize}
\item $g_0^n\to g_0$ in $L^p_{\rm loc}$,
\item $K^n\to K$ in $L^p_{\rm loc}$,
\item $(b^n,a^n,\nu^n)\to (b,a,\nu)$ in the sense that $A^nf \to Af$ locally uniformly on $\R^d\times\R^k$ for every $f\in C^2_c(\R^k)$, where $Af$ is defined in terms of the characteristic triplet by
\[
\begin{aligned}
Af(x,z) &= b(x)^\top \nabla f(z) + \frac12 \tr(a(x)\nabla^2f(z)) \\
&\quad + \int_{\R^k}(f(z+\zeta)-f(z)-\zeta^\top\nabla f(z))\nu(x,d\zeta),
\end{aligned}
\]
and $A^nf$ is defined analogously and is assumed to be continuous for every such $f$,
\item $(X^n,Z^n)\Rightarrow(X,Z)$ in $L^p_{\rm loc}\times D$.
\end{itemize}
Then $(X,Z)$ is a weak $L^p$ solution of \eqref{eq_SVEC} for the data $(g_0,K,b,a,\nu)$.
\end{theorem}

It is important to appreciate that no pointwise convergence of characteristic triplets is required in Theorem~\ref{T_stability}. For example, it may happen that $a^n=0$ for all $n$, but the limiting triplet has $a\ne0$. This is because diffusion can be approximated by small jumps, and we indeed make use of this in a crucial manner.

By combining the tightness and stability results with an approximation scheme for the characteristic triplet, we reduce the existence question to the pure jump case where $Z$ is piecewise constant with bounded jump intensity. A solution $X$ can then be constructed directly. The details are given in Section~\ref{S_pf_Tex1}.

At this point it is natural to ask about uniqueness of solutions to \eqref{eq_SVEC}. Standard counterexamples for SDEs reveal that no reasonable uniqueness statement will hold at the level of generality of Theorem~\ref{T_weak_existence}. Additional assumptions are needed. In Section~\ref{s:uniqueness} we prove a pathwise uniqueness theorem under suitable Lipschitz conditions; see Theorem~\ref{T_uniqueness}. This in turn yields uniqueness in law via the abstract machinery of \citet{K:14} and, as a by-product, strong existence. As for SDEs, uniqueness in the non-Lipschitz case is more delicate and not treated here. In certain situations, uniqueness in law can still be established; see for instance \cite{AJLP17} for the case of affine characteristics and continuous trajectories.

In Section~\ref{s:pathreg} we turn to path regularity of solutions $X$ of \eqref{eq_SVEC}. Basic examples show that $X$ can be as irregular as the kernel $K$ itself. However, often additional information is available that allows one to assert better path regularity. Criteria of this kind are collected in Theorem~\ref{T_pathreg}. 

At this stage let us mention various path regularity results for stochastic convolutions that already exist in the literature. For one-dimensional continuous kernels $K$, stochastic convolutions $\int_0^t K(t-s)dW_s$ with $W$ a standard Brownian motion may fail to be locally bounded in $t$ \citep[Theorem~1]{BPZ01}. However, under appropriate conditions on $K$, allowing in particular for certain singular kernels, a version with H\"older sample paths exists \citep[Lemma 2.4]{AJLP17}. If $W$ is replaced by a pure jump process,  \citet[Theorem~4]{R89} showed that the stochastic convolution fails to be locally bounded whenever the kernel is singular. Similar results appear in infinite dimensions, see \citet[Theorem 7.1]{BZ10}. Under additional regularity of the kernel, existence of H\"older continuous versions for fractional L\'evy processes has been established by \citet{M:06} and \citet{MN:11}.

Finally, in Section~\ref{S:applications} we sketch how our results can be applied to scaling limits of Hawkes processes (Subsection~\ref{S_app_1}) and approximations of solutions of \eqref{eq_SVEC} by means of finite-dimensional systems of Markovian SDEs (Subsection~\ref{S_app_2}).

Some basic auxiliary results are gathered in the appendix.

\section{Sobolev--Slobodeckij a priori estimate}\label{S_apriori}

This section is devoted to the proof of Theorem~\ref{T_apriori}. We will need the following inequality, taken from \citet[Theorem~1]{MR14}. It first appeared in \citet[Theorem 1]{N75}, but is also known as the {\em Bichteler--Jacod inequality} or {\em Kunita estimate}. We refer to \citet{MR14} for a historical survey of these maximal inequalities.

\begin{lemma}\label{L:novikovineq} Let $\mu$ be a random measure with compensator $\nu$, and define $\bar \mu=\mu-\nu$.  For any $T\in\R_+$ and $g$ such that the integral 
\[
M_t = \int_{[0,t)\times\R^k} g(s,\zeta) \bar \mu(ds,d\zeta)
\]
is well-defined for all $t\in[0,T]$, one has the inequality
\begin{equation*}
\begin{aligned}
\E\Big[ \sup_{t \leq T} \left|  M_t \right|^p\Big] \le C(p,T)\,\E\Big[ & \int_{[0,T)\times\R^k}|g(s,\zeta)|^p  \nu(ds,d\zeta) \\
&\quad +\Big(\int_{[0,T)\times\R^k}|g(s,\zeta)|^2   \nu(ds,d\zeta) \Big)^{p/2}\Big],
\end{aligned}
\end{equation*}
for any $p\geq 2$, where $C(p,T)$ only depends on $p$ and $T$.
\end{lemma}

We now proceed to the proof of Theorem~\ref{T_apriori}. Let therefore $d,k\in\N$, $p\in[2,\infty)$, and consider $(g_0,K,b,a,\nu)$ as in \ref{SVE_IC}--\ref{SVE_abnu}. We assume there exists a constant $c_{\rm LG}$ such that  \eqref{eq_LG} holds, and let $(X,Z)$ be a weak $L^p$ solution of \eqref{eq_SVEC} for the data $(g_0,K,b,a,\nu)$.

\begin{proof}[Proof of \eqref{eq:estimateLp}]
Observe that $Z$ admits the representation
\[
Z_t = \int_0^t b(X_s) ds + M^c_t + M^d_t, \quad t\ge0,
\]
where $M^c$ is a continuous local martingale with quadratic variation $\langle M^c \rangle = \int_0^\fdot a(X_s)ds$ and $M^d$ is a purely discontinuous local martingale whose jump measure has compensator $\nu(X,d\zeta)$. Define $\tau_n =\inf \{t\colon\int_0^t |X_s|^p ds \geq n\}\wedge T$.  Since $X$ is predictable with sample paths in $L^p_{\rm loc}$, the process $\int_0^\fdot |X_s|^p ds$ is continuous, adapted, and increasing. Thus $\tau_n$ is a stopping time for every $n$, and $\tau_n \to T$.  Define the process $X^n$ by $X^n_t=X_t\bm1_{t<\tau_n}$. We then have
\begin{align*}
&\|X^n  \|_{L^p(0,T)}^p \\
&\leq 4^{p-1}\Bigg(\|g_0\|_{L^p(0,T)}^p+\int_0^T \left|\int_{[0,t)}K(t-s)b(X^n_s)ds\right|^pdt\Bigg)\\
&\quad +4^{p-1}\Bigg(\int_0^T \left|\int_{[0,t)}K(t-s)dM^{c,n}_s\right|^pdt + \int_0^T \left|\int_{[0,t)}K(t-s)dM^{d,n}_s\right|^pdt\Bigg)\\
&=4^{p-1}\left(\|g_0\|_{L^p(0,T)}^p + \int_0^T (\mbox{\textbf{I}}_t+\mbox{\textbf{II}}_t+\mbox{\textbf{III}}_t)dt\right),
\end{align*}
where $M^{c,n}$ has quadratic variation equal to $\int_0^\fdot a(X_s^n) ds$, and the jump measure of $M^{d,n}$ has compensator $\nu(X^n,d\zeta)$. An application of the Jensen and BDG inequalities combined with Fubini's theorem and  \eqref{eq_LG} leads to 
\[
\E[\mbox{\textbf{I}}_t] + \E[\mbox{\textbf{II}}_t] \leq C(c_{\rm LG},p,T)  \int_0^t  |K(t-s)|^p (1+\E[|X^n_s |^p])ds,
\]
for {every $t \leq T$}. Thanks to  Novikov's inequality, see Lemma \ref{L:novikovineq}, we have
\begin{align*}
&\E[\mbox{\textbf{III}}_t] \le \E \left[\sup_{r \leq t} \left| \int_{{[0,r)}}K(t-s)dM_s^{d,n} \right|^p\right] \nonumber \\
&\leq  C(p,t)  \int_0^t  |K(t-s)|^p \E \left[\int_{\R^k}  |\zeta|^p \nu(X_s^n,d\zeta)  +  \left(\int_{\R^k}  |\zeta|^2  \nu(X_s^n,d\zeta)\right)^{p/2} \right]ds \nonumber \\
&\leq   C(c_{\rm LG},p,t)  \int_0^t  |K(t-s)|^p    (1+\E[|X_s^n|^p])ds, \nonumber 
\end{align*}
for {every $t \leq T$}, where the last inequality follows from \eqref{eq_LG}. Combining the above yields
\begin{align*}
\E[\|X^n&\|_{L^p(0,T)}^p]\leq C(c_{\rm LG},p,T)\\
&\times\Big(\|g_0\|_{L^p(0,T)}^p+\int_0^T\int_0^t |K(t-s)|^p(1+\E[|X^n_s|^p])ds\,dt\Big).
\end{align*}
Multiple changes of variables and applications of Tonelli's theorem yield
\begin{align*}
\int_0^T\int_0^t & |K(t-s)|^p(1+\E[|X^n_s|^p])ds\,dt \\
&=\int_0^T |K(s)|^p \int_0^{T-s} (1+\E[|X^n_{t}|^p])\,dt\,ds\\
&\leq T\|K\|_{L^p(0,T)}^p+\int_0^T |K(T-s)|^p\E[\|X^n\|_{L^p(0,s)}^p]\,ds.
\end{align*}
We deduce that the function $f_n(t)= \E[\|X^n\|_{L^p(0,t)}^p]$ satisfies the convolution inequality
\[
f_n(t) \le  C(c_{LG},p,T)\left(\|g_0\|_{L^p(0,T)}^p+ T\|K\|_{L^p(0,T)}^p\right)- (\widehat K * f_n)(t),
\]
where $\widehat K=- C(c_{LG},p,T)|K|^p$ lies in $L^1(0,T)$. The resolvent $\widehat R$ of $\widehat K$ is nonpositive and lies in $L^1(0,T)$; see \citet[Theorem~2.3.1 and its proof]{GLS:90}. Moreover, $f_n\le n$ by construction. Thus the Gronwall lemma for convolution inequalities applies; see Lemma~\ref{L_Gronwall}. In particular, we have
\[
f_n(T) \le C(c_{LG},p,T)\left(\|g_0\|_{L^p(0,T)}^p+T\|K\|_{L^p(0,T)}^p\right)\left(1+\|\widehat{R}\|_{L^1(0,T)}\right).
\]
As $n\to\infty$ we have $\tau_n\to T$, and hence {$f_n(T)\to \E[ \|X\|_{L^p(0,T)}^p ]$} by monotone convergence. We deduce \eqref{eq:estimateLp}, as desired. Finally, the continuous dependence on $K|_{[0,T]}$ follows from \citet[Theorem~2.3.1]{GLS:90}, which implies that the map from $L^p(0,T)$ to $\R$ that takes $K|_{[0,T]}$ to $\|\widehat{R}\|_{L^1(0,T)}$ is continuous.
\end{proof}

For the proof of the second part of Theorem \ref{T_apriori}, namely \eqref{eq:estimateW}, will need the following estimate.

\begin{lemma}\label{L_K_bounds}
	Let $K\colon\R_+\to\R^{d\times k}$ be measurable. For any $\eta>0$, $T\in\R_+$, $p \geq 2$ and nonnegative measurable function $f$, one has
	\begin{equation}\label{L_K_bounds_1}
	\begin{split}
	\int_0^T \int_0^T \int_{s\wedge t}^{s\vee t} \frac{|K(s\vee t-u)|^p}{|t-s|^{1+p\eta}} f(u) du\, ds\, dt \qquad \\
	\le  \|f\|_{L^1(0,T)} \frac1\eta \int_0^T |K(t)|^p t^{-\eta p}dt
	\end{split}
	\end{equation}
	and
	\begin{equation}\label{L_K_bounds_2}
	\begin{split}
	\int_0^T \int_0^T \int_0^{s\wedge t} \frac{ |K(t-u)-K(s-u)|^p }{|t-s|^{1+\eta p}} f(u) du\, ds\, dt \qquad \\
	\le \|f\|_{L^1(0,T)} \int_0^T \int_0^T \frac{|K(t)-K(s)|^p}{|t-s|^{1+\eta p}}ds\,dt.
	\end{split}
	\end{equation}
\end{lemma}

\begin{proof}
	We first prove \eqref{L_K_bounds_1}. Since $\int_{s\wedge t}^{s\vee t}(\ldots)du=\int_0^T (\bm1_{s<u<t}+\bm1_{t<u<s})(\ldots)du$, we may re-write the left-hand side of \eqref{L_K_bounds_1} as
	\[
	2\int_0^T \int_0^T \int_0^T \bm1_{s<u<t} \frac{|K(t-u)|^p}{(t-s)^{1+\eta p}} f(u) du\, ds\, dt.
	\]
	By Tonelli's theorem this equals
	\begin{align*}
	&2\int_0^T f(u) \int_u^T |K(t-u)|^p  \int_0^u \frac{1}{(t-s)^{1+\eta p}} ds\, dt\, du \\
	&\quad = \frac 2{\eta p} \int_0^T f(u) \int_u^T|K(t-u)|^p  ((t-u)^{-\eta p}-t^{-\eta p}) dt\, du \\
	&\quad \le\frac1\eta \int_0^T f(u) \int_u^T|K(t-u)|^p  (t-u)^{-\eta p} dt\, du.
	\end{align*}
	Since $\int_u^T|K(t-u)|^p  (t-u)^{-\eta p} dt = \int_0^{T-u} |K(v)|^p  v^{-\eta p} dv \le \int_0^T |K(v)|^p  v^{-\eta p} dv$, it follows that \eqref{L_K_bounds_1} holds, as claimed.
	
	We now prove \eqref{L_K_bounds_2}. Since $\int_0^{s\wedge t}(\ldots)du=\int_0^T \bm1_{u<s}\bm1_{u<t}(\ldots)du$, and by using Tonelli's theorem, we find that the left-hand side of \eqref{L_K_bounds_2} is equal to
	\[
	\int_0^T f(u) \int_u^T \int_u^T  \frac{ |K(t-u)-K(s-u)|^p }{|t-s|^{1+\eta p}} ds\, dt\, du.
	\]
	By a change of variables one sees that this is bounded by the right-hand side of \eqref{L_K_bounds_2}, as claimed.
\end{proof}

The proof of the second part of Theorem~\ref{T_apriori} is now straightforward. In addition to the above, we assume there exist a constant $\eta\in(0,1)$ and a locally bounded function $c_K\colon\R_+\to\R_+$ such that \eqref{eq_K_Wpnu} holds.

\begin{proof}[Proof of \eqref{eq:estimateW}]
Set $\bar X= X-g_0$ and observe that
\begin{align*}
|\bar X_t-\bar X_s| &\le \Big|\int_{[0,s\wedge t)}(K(t-u)-K(s-u))dZ_u\Big| \\
&\quad + \Big|\int_{[s\wedge t,s\vee t)}K(s\vee t-u)dZ_u\Big|,\quad \text{$\P \otimes dt\otimes ds$-a.e.}
\end{align*}
A similar argument as in the proof of \eqref{eq:estimateLp} shows that $\E[\int_0^T\int_0^T \frac{|\bar X_t- \bar X_s|^p}{|t-s|^{1+\eta p}}ds\,dt]$ is bounded above by 
\begin{multline*}
C(c_{LG},p,T)\Bigg(\int_0^T\int_0^T \int_{0}^{s\wedge t}\frac{|K(t-u)-K(s-u)|^p(1+\E[|X_u|^p])}{|t-s|^{1+\eta p}}du\,ds\,dt\\
+\int_0^T\int_0^T \int_{s\wedge t}^{s\vee t}\frac{|K(s\vee t-u)|^p(1+\E[|X_u|^p])}{|t-s|^{1+\eta p}}du\,ds\,dt\Bigg).
\end{multline*}
Applying \eqref{eq:estimateLp}, as well as Lemma~\ref{L_K_bounds} with $f(u)=1+\E[|X_u|^p]$, we obtain the bound \eqref{eq:estimateW} with a constant $c<\infty$ that depends on $d,k,p,\eta, c_K, c_{\rm LG},T$ as well as, $L^p$-continuously, on $K|_{[0,T]}$. Note that the set of restrictions $K|_{[0,T]}$ of kernels that satisfy \eqref{eq_K_Wpnu} with the given $c_K$ is relatively compact in $L^p(0,T)$. By maximizing the bound over all such $K$, we obtain a bound that only depends on $d,k,p,\eta, c_K, c_{\rm LG},T$.
\end{proof}

\section{Martingale problem and stability}\label{S_stability}

We consider initial conditions $g_0$ and convolution kernels $K$ as in \ref{SVE_IC}--\ref{SVE_K} of Section~\ref{S_main}, as well as linear operators $A$ that map functions $f\in C^2_c(\R^k)$ to measurable functions $Af\colon\R^d\times\R^k\to\R$, and satisfy the following growth bound for some $p\in[1,\infty)$:
\begin{equation} \label{lingrowth_MGP}
\begin{minipage}[c][3em][c]{.75\textwidth}
\begin{center}
For every $f\in C^2_c(\R^k)$ there is a finite constant $c_f$ such that $|Af(x,z)| \le c_f(1+|x|^p)$ for all $(x,z)\in\R^d\times \R^k$.
\end{center}
\end{minipage}
\end{equation}
Note that \eqref{lingrowth_MGP} ensures that $Af(\bm x,\bm z) \in L^1_{\rm loc}(\mathbb{R}_+, \mathbb{R})$ for any pair of functions $(\bm x,\bm z)\in L^p_{\rm loc}\times D$.

\begin{definition}\label{D_MP}
Let $p\in[1,\infty)$. A {\em solution of the local martingale problem} for $(g_0,K,A)$ is a pair $(X,Z)$ of processes with trajectories in $L^p_{\rm loc}\times D$, defined on a filtered probability space $(\Omega,\Fcal,\F,\P)$, such that $X$ is predictable, $Z$ is adapted with $Z_0=0$, the process
\begin{equation}\label{D_MP_2}
M^f_t = f(Z_t) - \int_0^t Af(X_s,Z_s)ds, \quad t\ge0,
\end{equation}
is a local martingale for every $f\in C^2_c(\R^k)$, and one has the equality
\begin{equation}\label{D_MP_1}
\int_0^t X_s ds = \int_0^t g_0(s)ds + \int_0^t K(t-s)Z_sds, \quad t\ge0.
\end{equation}
\end{definition}

Note that both the left- and right-hand sides of \eqref{D_MP_1} are continuous in $t$ and equal to zero for $t=0$. For the convolution $\int_0^t K(t-s)Z_sds$, this follows because $K$ is in $L^1_{\rm loc}$ and the trajectories of $Z$ are in $L^\infty_{\rm loc}$; see \citet[Corollary~2.2.3]{GLS:90}.

Our first goal is to establish the equivalence between weak $L^p$ solutions of \eqref{eq_SVEC} and solutions of the local martingale problem. The relevant operator $A$ is given by
\begin{equation}\label{eq:form1A}
\begin{aligned}
Af(x,z) &= b(x)^\top \nabla f(z) + \frac12 \tr(a(x)\nabla^2f(z)) \\
&\quad + \int_{\R^k}\left(f(z+\zeta)-f(z)-\chi(\zeta)^\top\nabla f(z)\right)\nu(x,d\zeta),
\end{aligned}
\end{equation}
where $(b,a,\nu)$ is the given characteristic triplet and $\chi$ is the truncation function. This equivalence will allow us to establish Theorem~\ref{T_stability} by proving a stability theorem for solutions of local martingale problems; see Theorem~\ref{T_stability_MGP} below. The latter is easier, because the conditions \eqref{D_MP_2} and \eqref{D_MP_1} are more easily shown to be closed with respect to suitable perturbations of $X$, $Z$, $g_0$, $K$ and $A$.

\begin{lemma}\label{L:Fubini}
Let $p \in [2, \infty)$. Consider a kernel $K\in L^2_{\rm loc}$ and a characteristic triplet $(b,a,\nu)$ satisfying \eqref{eq_GB_1}. Let $X$ be a predictable process with trajectories in $L^p_{\rm loc}$ and let $Z$ be an It\^o semimartingale whose differential characteristics with respect to some given truncation function $\chi$ are $b(X),a(X),\nu(X,d\zeta)$.
 Then $\int_{[0,t)}K(t-s)dZ_s$ is well-defined for almost every $t\in\R_+$, and
\[
\int_0^t \left(\int_{[0,s)}K(s-u)dZ_u\right) ds = \int_0^t K(t-s)Z_sds,\quad t\ge 0.
\]
\end{lemma}

\begin{proof}
The stochastic integral $\int_{[0,t)}K(t-s)dZ_s$ is well-defined for a.e.\ $t\in\R_+$ by Lemma~\ref{L_SI_welldef}. Define $\kappa(x)=|b(x)| + |a(x)| + \int_{\R^k} (1\wedge |\zeta|^2) \nu(x,d\zeta)$. The bound \eqref{eq_GB_1} and a change of variables yield
\begin{align*}
\int_0^{t}&\left(\int_0^{t} |K(s-u)|^2{\rm 1}_{\{u< s\}}d s\right) \kappa(X_u) du\\
&\leq c \int_0^{t}\left(\int_0^{t} |K(v)|^2{\rm 1}_{\{v< t-u\}}d v\right) (1+|X_u|^p) du\\
&\leq c \|K\|^2_{L^2(0,t)}(t+\|X\|^p_{L^p(0,t)})<\infty.
\end{align*}
This implies that the stochastic integral 
\[
\int_0^t \left(\int_0^{t} |K(s-u)|^2{\rm 1}_{\{u< s\}}d s\right)^{\frac12} dZ_u = \int_0^t \left(\int_0^{t-u} |K(v)|^2 d v\right)^{\frac12} dZ_u
\] 
is well defined. Two applications of the stochastic Fubini theorem in \citet[Theorem 65]{P05} along with two changes of variables yield
\begin{align*}
\int_0^t \left(\int_{[0,s)}K(s-u)dZ_u\right) ds&=\int_0^t \left(\int_0^t K(s-u){\rm 1}_{\{u<s\}} ds\right) dZ_u \\
&=\int_0^t  \left(\int_0^{t-u} K(v) dv\right) dZ_u \\
&=\int_0^t K(v)\left(\int_0^{t-v} dZ_u  \right)  dv\\
&=\int_0^t K(t-s) Z_{s} ds.
\end{align*}
This completes the proof.
\end{proof}

We can now prove the equivalence of weak $L^p$ solutions and solutions of the local martingale problem.

\begin{lemma}\label{L_SVE_iff_MGP}
Let $p \in [2, \infty)$ and consider data $(g_0,K,b,a,\nu)$ as in \ref{SVE_IC}--\ref{SVE_abnu} and a truncation function $\chi$. A pair $(X,Z)$ is a weak $L^p$ solution of \eqref{eq_SVEC} if and only if it is a solution of the local martingale problem for $(g_0,K,A)$, where $A$ is given by \eqref{eq:form1A}.
\end{lemma}

\begin{proof}
Suppose first $(X,Z)$ is a weak $L^p$ solution of \eqref{eq_SVEC}. It\^o's formula applied to $Z$ shows that the process $M^f$ in \eqref{D_MP_2} is a local martingale for every $f\in C^2_c(\R^k)$; see \citet[Theorem II.2.42 (a)$\Rightarrow$(c)]{jac_shi_03}. Furthermore, integrating both sides of \eqref{eq_SVEC_def} and invoking Lemma~\ref{L:Fubini} yields \eqref{D_MP_1}. Thus $(X,Z)$ is a solution of the local martingale problem.

Conversely, suppose $(X,Z)$ is a solution of the local martingale problem for $(g_0,K,A)$. Lemma \ref{L:Fubini} and \eqref{D_MP_1} yield
\[
\int_0^T X_t dt =\int_0^T\left( g_0(t) + \int_{[0,t)}K(t-s)dZ_s\right)dt
\]
for any $T>0$. This implies \eqref{eq_SVEC_def}. It remains to check that $Z$ is a semimartingale with differential characteristics $b(X),a(X),\nu(X,d\zeta)$ with respect to $\chi$. This will follow from \citet[Theorem II.2.42 (c)$\Rightarrow$(a)]{jac_shi_03}, once we prove that $M^f$ given in \eqref{D_MP_2} is a local martingale not only for all $f\in C^2_c(\R^k)$, but for all $f\in C^2_b(\R^k)$, i.e.~bounded functions which are continuously twice differentiable. Observe that $M^f$ remains well-defined thanks to \eqref{eq_GB_1}. We adapt the proof of \citet[Proposition~3.2]{CFY05}. Consider the stopping times
	\begin{align*}
	T_m &= \inf \{t \geq 0\colon \int_0^t (1+|X_s|^p)ds \geq m \},\\
	S_m &= \inf \{t \geq 0\colon |Z_{t-}| \geq m \mbox{ or }  |Z_{t}| \geq m\},\\
	\tau_m &= T_m\wedge S_m,
	\end{align*}
for $m \geq 1$. It is clear that $\tau_m\to\infty$ as $m\to\infty$. Fix any function $f\in C_b^2(\R^k)$. Fix also functions $\varphi_n\in C^2_c(\R^k)$ taking values in $[0,1]$ and equal to one on the centered ball $B(0,n)$ of radius $n$. Then $f\varphi_n \in C^2_c(\R^k)$, so that $M^{f\varphi_n}$ defined as in \eqref{D_MP_2} is a local martingale for each $n$. Write $M^{n,m}_t=M^{f\varphi_n}_{t\wedge \tau_m}$. We then have for $n,m \in \N$ 
\begin{align*}
|M^{n,m}_t| \leq  \|f\|_{\infty} + m c_n, \quad t \geq 0,  
\end{align*}  
where the constant $c_n$ comes from \eqref{eq_GB_1} and depends on $n$. Hence, $M^{n,m}$ is a true martingale for each $m,n\in\N$. 
Fix $m\in\N$ and set $M^m_t=M^f_{t\wedge\tau_m}$. For all $n > m$, by  definition of $T_m$ and the fact that $\varphi_n=1$ on $B(0,n)$ we have 
\begin{align*}
M^m_t -M^{n,m}_t &=    \int_{(0,t\wedge \tau_m]\times\R^k}  \left(f(Z_s + \zeta ) -(f\varphi_n)(Z_s + \zeta )\right) \nu(X_s,d\zeta)ds.
\end{align*}
Thus
\begin{align*}
|M^m_t -M^{n,m}_t |
&\leq     \|f\|_{\infty} \int_{(0,t\wedge \tau_m]\times\R^k}  \bm 1_{|\zeta| \geq n-m} \nu(X_s,d\zeta) ds.
\end{align*}
As $n \to \infty$, the right--hand side  tends to zero  in $L^1(\P)$, by virtue of the dominated convergence theorem. Indeed,  $1\wedge|\zeta|^2\ge \bm 1_{|\zeta| \geq n-m}\to0$ as $n\to\infty$, and it follows from  \eqref{eq_GB_1}  that 
\[
\int_{(0,t\wedge \tau_m]\times\R^d}  (1\wedge|\zeta|^2) \nu(X_s,d\zeta) ds \leq c  \int_{0}^{t\wedge \tau_m}  (1+|X_s|^p) ds \leq cm.
\]
We conclude that $\E[|M_t^{n,m} - M_t^m|]\to 0$ as $n\to \infty$. Thus $M^m_t=M^f_{t\wedge\tau_m}$ is a martingale being an $L^1(\P)$-limit of martingales. Thus $M^f$ is a local martingale, as required.
\end{proof}

The following is our main result on stability for solutions of local martingale problems. Together with Lemma~\ref{L_SVE_iff_MGP}, it will imply Theorem~\ref{T_stability}. We let $\bm x=(\bm x(t))_{t\ge0}$ and $\bm z=(\bm z(t))_{t\ge0}$ denote generic elements of $L^p_{\rm loc}$ and $D$, respectively.

\begin{theorem}\label{T_stability_MGP}
Let $d,k\in\N$, $p\in(1,\infty)$. Consider data $(g_0^n,K^n,A^n)$ for $n\in\N$ and $(g_0,K,A)$, and assume that the $A^n$ satisfy \eqref{lingrowth_MGP} with constants $c_f$ that do not depend on $n$, and that $A^nf$ is continuous for every $f \in C^2_c(\R^k)$. For each $n$, let $(X^n,Z^n)$ be a solution of the local martingale problem for $(g_0^n,K^n,A^n)$. Assume that:
\begin{itemize}
\item $g_0^n\to g_0$ in $L^p_{\rm loc}$,
\item $K^n\to K$ in $L^p_{\rm loc}$,
\item $A^nf\to Af$ locally uniformly on {$\mathbb{R}^d \times \mathbb{R}^k$} for every $f\in C^2_c(\R^k)$,
\item $(X^n,Z^n)\Rightarrow(X,Z)$ in $L^p_{\rm loc}\times D$ for some limiting process $(X,Z)$.
\end{itemize}
Then $(X,Z)$ is a solution of the local martingale problem for $(g_0,K,A)$.
\end{theorem}

\begin{proof}
Let $(\Omega^n,\Fcal^n,(\Fcal^n_t)_{t\ge0},\P^n)$ be the filtered probability space where $(X^n,Z^n)$ is defined. We may assume without loss of generality that this space supports an $\Fcal^n_0$-measurable standard uniform random variable $U^n$ that is independent of $(X^n,Z^n)$. We then have $(U^n,X^n,Z^n)\Rightarrow(U,X,Z)$ in $[0,1]\times L^p_{\rm loc}\times D$, where $U$ is standard uniform and independent of $(X,Z)$. The standard uniform random variable $U$ will be used below as a randomization device to avoid the jumps of $Z$.

Fix $f\in C^2_c(\R^k)$ and $m\in\N$. For any $(u,\bm x)\in[0,1]\times L^p_{\rm loc}$, define
\[
\tau(u,\bm x)=\inf\{t\ge0\colon \int_0^t(u+1+|\bm x(s)|^p ) ds \ge m\}.
\]
Then $\tau(U^n,X^n)$ is a stopping time in $(\Fcal^n_t)_{t\ge0}$, and the growth bound \eqref{lingrowth_MGP} yields
\[
\int_0^{t\wedge\tau(U^n,X^n)} |A^nf(X^n_s,Z^n_s)|ds \le c_f \int_0^{t\wedge\tau(U^n,X^n)} (1+|X^n_s|^p)ds \le m c_f.
\]
Thus the local martingale
\[
M^n_t = f(Z^n_{t\wedge\tau(U^n,X^n)}) - \int_0^{t\wedge\tau(U^n,X^n)} A^nf(X^n_s,Z^n_s)ds, \quad t\ge0,
\]
satisfies
\begin{equation}\label{stability_pf2_new}
|M^n_t| \le \|f\|_\infty + m c_f, \quad t\ge0.
\end{equation}
In particular it is a true martingale, so for any time points $0\le t_1 < \cdots < t_k \le s < t$, and functions $h \in C([0,1])$ and $g_i\in C_b(\R^d\times\R^d)$, $i=1,\ldots,k$, we have
\begin{equation}\label{stability_pf1_new}
\E\left[ (M^n_t-M^n_s) h(U^n) \prod_{i=1}^k g_i\left(\int_0^{t_i} X^n_r dr, Z^n_{t_i}\right) \right] = 0,
\end{equation}
where $\E$ is understood as expectation under $\P^n$.

Next, by Skorokhod's representation theorem (see \citet[Theorem~6.7]{B:99}), we may assume that all the triplets $(U^n,X^n,Z^n)$ and $(U,X,Z)$ are defined on a common probability space $(\Omega,\Fcal,\P)$, that $(U^n,X^n,Z^n)\to(U,X,Z)$ in $[0,1]\times L^p_{\rm loc}\times D$ almost surely, and that each triplet has the same law under $\P$ as it did under $\P^n$.\footnote{We can however {\em not} assume that the filtrations $(\Fcal^n_t)_{t\ge0}$ are the same.} In particular, \eqref{stability_pf1_new} still holds, now with $\E$ understood as expectation under $\P$.

We now prepare to pass to the limit in \eqref{stability_pf1_new}. One easily checks that the map $(u,\bm x)\mapsto\tau(u,\bm x)$ is continuous. Combined with Lemma~\ref{L_conv_int} below, it follows that
\[
\int_0^{t\wedge\tau(U^n,X^n)} A^nf(X^n_r,Z^n_r)dr \to \int_0^{t\wedge\tau(U,X)} Af(X_r,Z_r)dr
\]
almost surely for any $t\ge0$. Moreover, $Z$ is continuous at $\tau(U,X)$, almost surely. To see this, let $\{T_i(\bm z)\colon i\in\N\}$ denote an enumeration of the countably many jump times of the function $\bm z\in D$. We choose $T_i(\bm z)$ measurable in $\bm z$. Since $U$ and $(X,Z)$ are independent, and since for any $\bm x\in L^p_{\rm loc}$ the law of $\tau(U,\bm x)$ has no atoms, we get
\[
\P(\tau(U,X)=T_i(Z)) = \E\left[ \P(\tau(U,\bm x) = T_i(\bm z))|_{(\bm x,\bm z)=(X,Z)}\right] = 0.
\]
Thus
\[
\P(\tau(U,X)\in \{T_i(\bm z)\colon i\in\N\}) \le \sum_{i\in\N} \P(\tau(U,X)=T_i(Z)) = 0,
\]
showing that $Z$ is indeed continuous at $\tau(U,X)$, almost surely. We conclude that
\[
M^n_t \to M_t
\]
almost surely for any $t\in\Ccal(Z)=\{r\in\R_+\colon\P(Z_r=Z_{r-})=1\}$, where we define
\[
M_t = f(Z_{t\wedge\tau(U,X)}) - \int_0^{t\wedge\tau(U,X)} Af(X_s,Z_s)ds, \quad t\ge0.
\]
Selecting $0\le t_1<\ldots<t_k\le s<t$ from $\Ccal(Z)$, we may thus use the bounded convergence theorem, justified by \eqref{stability_pf2_new}, to pass to the limit in \eqref{stability_pf1_new} to obtain
\begin{equation}\label{stability_pf3_new}
\E\left[ (M_t-M_s) h(U) \prod_{i=1}^k g_i\left(\int_0^{t_i} X_r dr, Z_{t_i}\right) \right] = 0.
\end{equation}
By \citet[Theorem~3.7.7]{EK:05}, $\Ccal(Z)$ is dense in $\R_+$. Along with right-continuity of $M$ and $Z$, this implies that \eqref{stability_pf3_new} actually holds for any choice of times points $0\le t_1<\ldots<t_k\le s<t$. Thus $M$ is a martingale with respect to the filtration given by
\[
\Fcal_t=\sigma(U)\vee\sigma(\int_0^s X_r dr,Z_s\colon s\le t), \quad t\ge0.
\]
Since $\tau(U,X)$ is a stopping time for this filtration, and since the constant $m$ in the definition of $\tau(U,X)$ was arbitrary, the process $M^f$ in \eqref{D_MP_2} is a local martingale.

We must also verify \eqref{D_MP_1}. This is immediate from $L^p$ convergence of $g_0^n$ and $X^n$ as well as Lemma~\ref{L_conv_K} below. This lets us pass to the limit in the identity $\int_0^t X^n_sds = \int_0^t g_0^n(s)ds + \int_0^t K^n(t-s)Z^n_sds$, which is valid by assumption. 

It only remains to ensure that $Z$ is adapted and $X$ is predictable. Adaptedness of $Z$ holds by definition of the filtration. It is however not clear that $X$ is predictable. Therefore, we replace $X$ by the process $\widetilde X=\liminf_{h\downarrow0} \widetilde X^h$, where for each $h>0$ we define
\[
\widetilde X^h_t = \frac1h \int_{(t-h)\vee 0}^t X_s ds, \quad t\ge0.
\]
Note that $\widetilde X$ is predictable, being the pointwise liminf of the continuous and adapted processes $\widetilde X^h$. Moreover, for every fixed $\omega$, the trajectory $\widetilde X(\omega)$ coincides with $X(\omega)$ almost everywhere by Lebesgue's differentiation theorem. Replacing $X$ by $\widetilde X$ therefore does not affect either \eqref{D_MP_1} or the local martingale property in \eqref{D_MP_2}.
\end{proof}

The following two lemmas were used in the proof of Theorem~\ref{T_stability_MGP}. The first one uses the convolution notation $(f*g)(t)=\int_0^t f(t-s)g(s)ds$.

\begin{lemma}\label{L_conv_K}
Fix $p\in(1,\infty)$. If $K^n\to K$ in $L^p_{\rm loc}(\R_+,\R^{d\times k})$ and $\bm z_n\to\bm z$ in $D$, then $K^n*\bm z_n\to K*\bm z$ locally uniformly.
\end{lemma}

\begin{proof}
Fix any $T\in\R_+$ and let $q\in(1,\infty)$ satisfy $p^{-1}+q^{-1}=1$. The triangle inequality and Young's inequality, see Lemma~\ref{L_Young} with $r=\infty$, give
\begin{align*}
\| K*\bm z &- K^n*\bm z_n\|_{L^\infty(0,T)}\\
&\le\| K*(\bm z-\bm z_n)\|_{L^\infty(0,T)} + \| (K- K^n)*\bm z_n\|_{L^\infty(0,T)} \\
&\le\|K\|_{L^p(0,T)}\|\bm z-\bm z_n\|_{L^q(0,T)} + \|K- K^n\|_{L^p(0,T)}\|\bm z_n\|_{L^q(0,T)}.
\end{align*}
Since $\bm z_n\to\bm z$ in $D$, we have $\sup_n\|\bm z-\bm z_n\|_{L^\infty(0,T)}<\infty$ and $\bm z_n(t)\to\bm z(t)$ for almost every $t\in[0,T]$. Hence $\bm z_n\to\bm z$ in $L^q(0,T)$ by the dominated convergence theorem. Since $K*\bm z$ and $K^n*\bm z_n$ are continuous functions due to \citet[Corollary~2.2.3]{GLS:90}, the $L^\infty(0,T)$ norm coincides with the supremum norm on $[0,T]$. The result follows.
\end{proof}

\begin{lemma}\label{L_conv_int}
Fix $d,k\in\N$, $p\in[1,\infty)$. Let $g_n\colon\R^d\times\R^k\to\R$ be continuous functions satisfying the following polynomial growth condition: For every compact subset $Q\subset\R^k$, there exists a constant $c_Q\in\R_+$ such that
\begin{equation}\label{L_cont_lingrowth}
|g_n(x,z)| \le c_Q(1+|x|^p), \quad (n,x,z)\in\N\times\R^d\times Q.
\end{equation}
Assume that $g_n\to g$ locally uniformly for some function $g\colon\R^d\times\R^k\to\R$. Then, whenever $(\bm x_n,\bm z_n)\to(\bm x,\bm z)$ in $L^p_{\rm loc}\times D$, we have
\[
\int_0^t g_n(\bm x_n(s),\bm z_n(s)) ds \to \int_0^t g(\bm x(s),\bm z(s)) ds
\]
locally uniformly in $t\in\R_+$.
\end{lemma}

\begin{proof}
Suppose $(\bm x_n,\bm z_n)\to(\bm x,\bm z)$ in $L^p_{\rm loc}\times D$. Fix $T\in\R_+$, let $Q\subset\R^d$ be a compact set that contains the values attained by $\bm z$ and $\bm z_n$, $n\in\N$, over $[0,T]$, and let $c_Q$ be the associated constant in \eqref{L_cont_lingrowth}. Let $R\in[1,\infty)$ be an arbitrary constant, and write
\begin{align*}
\int_0^T &|g_n(\bm x_n(s),\bm z_n(s))- g(\bm x(s),\bm z(s))| ds \\
&\le \int_0^T |g_n(\bm x_n(s),\bm z_n(s))- g(\bm x_n(s),\bm z_n(s))|\bm1_{|\bm x_n(s)|\le R} \,ds\\
&\quad +\int_0^T |g_n(\bm x_n(s),\bm z_n(s))- g(\bm x_n(s),\bm z_n(s))|\bm1_{|\bm x_n(s)|> R} \,ds \\
&\quad +\int_0^T |g(\bm x_n(s),\bm z_n(s))- g(\bm x(s),\bm z_n(s))|\bm1_{|\bm x_n(s)|\vee|\bm x(s)| \le R} \,ds \\
&\quad +\int_0^T |g(\bm x_n(s),\bm z_n(s))- g(\bm x(s),\bm z_n(s))|\bm1_{|\bm x_n(s)|\vee|\bm x(s)| > R} \,ds \\
&\quad + \int_0^T |g(\bm x(s),\bm z_n(s))- g(\bm x(s),\bm z(s))|ds \\
&= {\bf I}_n + {\bf II}_n + {\bf III}_n + {\bf IV}_n + {\bf V}_n.
\end{align*}

We bound these terms individually. First, defining the compact set $Q_R= \overline{B(0,R)} \times  Q$, where $\overline{B(0,R)}=\{x\in\R^d\colon|x|\le R\}$ is the centered closed ball of radius $R$, we have
\[
{\bf I}_n \le T  \sup_{(x,z)\in Q_R}|g_n(x,z)-g(x,z)| \to 0 \quad (n\to\infty).
\]

Next, consider the restrictions $\bm x_n|_{[0,T]}$, again denoted by $\bm x_n$ for simplicity; they are convergent in $L^p(0,T)$. The Vitali convergence theorem implies that $\{|\bm x_n|^p\colon n\in\N\}$ is uniformly integrable. Since $g$ satisfies the same polynomial growth condition \eqref{L_cont_lingrowth} as the $g_n$ and since $R \geq 1$, we then get
\[
{\bf II}_n \le 4c_Q \int_0^T |\bm x_n(s)|^p\bm1_{|\bm x_n(s)| > R} \,ds \le \varphi_{\bf II}(R^p),
\]
where $\varphi_{\bf II}(R^p)=4c_Q \sup_n\int_0^T |\bm x_n(s)|^p\bm1_{|\bm x_n(s)|^p > R^p} \,ds$ converges to zero as $R\to\infty$ by the definition of uniform integrability. In a similar manner, we get
\[
{\bf IV}_n \le 4c_Q \int_0^T (|\bm x_n(s)|\vee |\bm x(s)|)^p\bm1_{|\bm x_n(s)|\vee|\bm x(s)| > R} \,ds \le \varphi_{\bf IV}(R^p),
\]
where $\varphi_{\bf IV}(R^p)=4c_Q \sup_n\int_0^T (|\bm x_n(s)|\vee |\bm x(s)|)^p\bm1_{|\bm x_n(s)|\vee|\bm x(s)| > R} \,ds$ also converges to zero as $R\to\infty$.

We now turn to ${\bf III}_n$. Let $\omega_R\colon\R_+\to\R_+$ be a continuous strictly increasing concave function with $\omega_R(0)=0$ such that
\[
\sup_{z\in Q}|g(x,z)-g(y,z)| \le \omega_R(|x-y|), \quad x,y \in \overline{B(0,R)}.
\]
Such a function exists because $g$ is uniformly continuous on the compact set $Q_R$. Its inverse $\omega_R^{-1}$ exists and is convex, so by using Jensen's inequality we get
\begin{align*}
{\bf III}_n &\le \int_0^T \omega_R(|\bm x_n(s)-\bm x(s)|)\,ds \\
&=T\, \omega_R \circ \omega_R^{-1}\left( \int_0^T \omega_R(|\bm x_n(s)-\bm x(s)|)\,\frac{ds}{T} \right) \\
&\le T\, \omega_R\left( \int_0^T |\bm x_n(s)-\bm x(s)|\,\frac{ds}{T} \right) \\
&\to 0 \quad (n\to\infty).
\end{align*}

Finally, consider ${\bf V}_n$. Since $\bm z_n\to\bm z$ in $D$, we have $\bm z_n(s)\to\bm z(s)$ for almost every $s\in\R_+$. Thus the integrand in ${\bf V}_n$ converges to zero for almost every $s\in\R_+$. Moreover, the polynomial growth condition \eqref{L_cont_lingrowth} implies that the integrand is bounded by $2c_Q(1+|\bm x(s)|^p)$, which has finite $L^1([0,T],\R^d)$-norm. The dominated convergence theorem now shows that ${\bf V}_n\to0$ as $n\to\infty$.

Combining the above bounds, we obtain
\[
\limsup_{n\to\infty} \int_0^T |g_n(\bm x_n(s),\bm z_n(s))- g(\bm x(s),\bm z(s))| ds \le \varphi_{\bf II}(R^p) + \varphi_{\bf IV}(R^p).
\]
Sending $R$ to infinity shows that the left-hand side is actually equal to zero. This completes the proof.
\end{proof}

The proof of Theorem~\ref{T_stability} is now straightforward.

\begin{proof}[Proof of Theorem~\ref{T_stability}]
This is a consequence of Lemma~\ref{L_SVE_iff_MGP} and Theorem~\ref{T_stability_MGP}. We only need to observe that the ``truncation function'' $\chi(\zeta)=\zeta$ can be used under the stronger integrability condition \eqref{eq_LG}, and that the $A^n$ satisfy \eqref{lingrowth_MGP} with constants $c_f$ that do not depend on $n$. To see this, observe that any $f\in C^2_c(\R^k)$ satisfies
\begin{align}\label{eq_MGP_pf_existence_3}
|f(z+\zeta)-f(z)-\zeta^\top\nabla f(z)| \le  \frac12 \|\nabla^2f\|_\infty |\zeta|^2.
\end{align}
Therefore,
\begin{align*}
|A^n f(x,z)| &\le  \left( \|\nabla f\|_\infty + \frac12\|\nabla^2f\|_\infty \right)\\
&\quad\times\left( |b^n(x)| +  |a^n(x)| + \int_{\R^k} |\zeta|^2 \nu^n(x,d\zeta) \right).
\end{align*}
Since $(b^n,a^n,\nu^n)$ satisfy \eqref{eq_LG} with a common constant $c_{\rm LG}$, and due to the bounds $|b^n(x)| \le 1+|b^n(x)|^2$ and $|x|^2\le1+|x|^p$, we deduce that $|A^n f(x,z)| \le c_f (1+|x|^p)$ holds with
\begin{align}\label{eq_constant_cf} 
c_f =  2 (1+ c_{\rm LG})\left( \|\nabla f\|_\infty + \frac12 \|\nabla^2f\|_\infty \right).
\end{align}
This does not depend on $n$, as required. The proof is complete.
\end{proof}

\section{Existence of weak $L^p$ solutions}\label{S_pf_Tex1}

This section is devoted to the proof of Theorem~\ref{T_weak_existence}. We first give an elementary existence result for the simple pure jump case where the diffusion part of the characteristic triplet vanishes, and the jump kernel is uniformly bounded.

\begin{lemma}\label{L_basic_existence}
Let $K\colon\R_+\to\R^{d\times k}$ and $g_0\colon\R_+\to\R^d$ be measurable functions. Let $\nu(x,d\zeta)$ be a bounded kernel from $\R^d$ into $\R^k$, meaning that $\sup_{x\in\R^d}\nu(x,\R^k)<\infty$. Then there exists a filtered probability space with a predictable process $X$ and a c\`adl\`ag piecewise constant semimartingale $Z$ such that
\[
X_t = g_0(t) + \int_{[0,t)} K(t-s) dZ_s, \quad t\ge0,
\]
and the differential characteristics of $Z$ are $b(X)=\int_{\R^k}\zeta\nu(X,d\zeta)$, $a(X)=0$, $\nu(X,d\zeta)$.
\end{lemma}

\begin{proof}
Let $\{(U_n,E_n)\colon n\in\N\}$ be a collection of independent random variables on a probability space $(\Omega,\Fcal,\P)$, with $U_n$ standard uniform and $E_n$ standard exponential. Define
\[
T_0 = 0, \quad X^0_t = g_0(t), \quad Z^0_t=0, \quad t\ge0.
\]
We now construct processes $X^n$, $Z^n$ and random times $T_n$ recursively as follows. For each $n\in\N$, if $X^{n-1}$ and $Z^{n-1}$ have already been constructed, define a jump time $T_n$ and jump size $J_n$ as follows. First set
\[
T_n = \inf\{t>T_{n-1}\colon \int_{T_{n-1}}^t \nu(X^{n-1}_s,\R^k) ds \ge E_n\},
\]
and note that $T_n>T_{n-1}$ since the kernel $\nu(x,d\zeta)$ is bounded. Then let $F\colon\R^d\times[0,1]\to\R^k$ be a measurable function with the following property: If $U$ is standard uniform, then $F(x,U)$ has distribution $\nu(x,\fdot)/\nu(x,\R^k)$ if $\nu(x,\R^k)>0$, and $F(x,U)=0$ otherwise. Set $J_n=F(X^{n-1}_{T_n},U_n)$. We can now define
\begin{align*}
X^n_t &= X^{n-1}_t + K(t-T_n)J_n \bm1_{t>T_n} \\
Z^n_t &= Z^{n-1}_t + J_n \bm1_{t\ge T_n}
\end{align*}
for $t\ge0$. Note that $(X^n,Z^n)$ coincides with $(X^{n-1},Z^{n-1})$ on $[0,T_n)$. 

Since the kernel $\nu(x,d\zeta)$ is bounded, we have $\sup_{x\in\R^d}\nu(x,\R^k)\le c$ for some constant $c$, and thus $T_n-T_{n-1}\ge \inf\{t>0\colon ct\ge E_n\}=E_n/c$. It follows from the Borel--Cantelli lemma that $\lim_{n\to\infty}T_n=\sum_{n\in\N}(T_n-T_{n-1})=\infty$. We can thus define $(X_t,Z_t)$ for all $t\ge0$ by setting $(X_t,Z_t)=(X^n_t,Z^n_t)$ for $t<T_n$. It follows from the construction that $Z$ is c\`adl\`ag and piecewise constant, and that
\[
X_t = g_0(t) + \sum_{n\colon t>T_n} K(t-T_n) \Delta Z_{T_n}, \quad t\ge0.
\]
This is the desired convolution equation.

Let $(\Fcal_t)_{t\ge0}$ be the filtration generated by $Z$, so that in particular $Z$ is a semimartingale. It follows from the construction of $Z$ that its jump characteristic is $\nu(X_t,d\zeta)dt$, provided $X$ is predictable. We now show that this is the case. Indeed, any process of the form $f(t)g(T_n,J_n)\bm1_{t>T_n}$ is predictable, so by a monotone class argument the same is true for $K(t-T_n)J_n \bm1_{t>T_n}$. Since $X^0=g_0$ is predictable, it follows by induction that $X^n$ is predictable for each $n$. Thus $X$ is predictable, and the proof is complete.
\end{proof}

We now proceed with the proof of Theorem~\ref{T_weak_existence}. Throughout the rest of this section, we therefore consider $d,k\in\N$, $p\in[2,\infty)$, and $(g_0,K,b,a,\nu)$ as in \ref{SVE_IC}--\ref{SVE_abnu}. We assume that $b$ and $a$ are continuous, and that $x\mapsto |\zeta|^2 \nu(x,d\zeta)$ is continuous from $\R^d$ to $M_+(\R^k)$, the finite positive measures on $\R^k$ with the topology of weak convergence. We also assume there exist a constant $\eta\in(0,1)$, a locally bounded function $c_K\colon\R_+\to\R_+$, and a constant $c_{\rm LG}$ such that \eqref{eq_K_Wpnu} and \eqref{eq_LG} hold.

Lemma~\ref{L_SVE_iff_MGP} connects \eqref{eq_SVEC} to the local martingale problem for $(g_0,K,A)$, where the operator $A$ is given by
\begin{equation}\label{eq_MGP_pf_existence}
\begin{aligned}
Af(x,z) &= b(x)^\top \nabla f(z) + \frac12 \tr(a(x)\nabla^2f(z)) \\
&\quad + \int_{\R^k}(f(z+\zeta)-f(z)-\zeta^\top\nabla f(z))\nu(x,d\zeta).
\end{aligned}
\end{equation}
By the same arguments as in the proof of Theorem~\ref{T_stability}, the inequality \eqref{eq_MGP_pf_existence_3} and the  growth bound \eqref{eq_LG}, 
 $A$ satisfies \eqref{lingrowth_MGP} with the constants $c_f$ given by \eqref{eq_constant_cf}. In the following lemma, we construct approximations of $A$.

\begin{lemma}\label{L_approxA}
Let $A$ be as in~\eqref{eq_MGP_pf_existence}. Then there exist kernels $\nu^n(x,d\zeta)$ from $\R^d$ into $\R^k$ with the following properties.
\begin{enumerate}
\item\label{L_approxA_1} boundedness and compact support: $\sup_{x\in\R^d}\nu^n(x,\R^k)<\infty$, and $\nu^n(x,\fdot)$ is compactly supported for every $x\in\R^d$,
\item\label{L_approxA_2} linear growth uniformly in $n$: with $b^n(x)=\int_{\R^k}\zeta\nu^n(x,d\zeta)$, one has
\begin{equation}\label{L_approxA_2_LG}
|b^n(x)|^2 + \int_{\R^k} |\zeta|^2 \nu^n(x,d\zeta) + \left(\int_{\R^k}|\zeta|^p \nu^n(x,d\zeta)\right)^{2/p} \le c_{\rm LG}' (1+|x|^2),
\end{equation}
where $c'_{\rm LG}= (5+2\sqrt{d})c_{\rm LG}$,
\item\label{L_approxA_3} locally uniform approximation: for every $f\in C^2_c(\R^k)$, defining
\[
A^nf(x,z)=\int_{\R^k}(f(z+\zeta)-f(z))\nu^n(x,d\zeta),
\]
we have $A^nf\in C(\R^d\times\R^k)$ and $A^nf\to Af$ locally uniformly.
\end{enumerate}
\end{lemma}

\begin{proof}
Multiplying by a continuous cutoff function if necessary, we may assume that $b(x)$, $a(x)$, and $\nu(x,d\zeta)$ are zero for all $x$ outside some compact set $Q$. Moreover, we can approximate the $b$, $a$, and $\nu$ parts separately and then add up the approximations (observing that the left-hand side of \eqref{L_approxA_2_LG} is subadditive in $(b^n,\nu^n)$, so that we may simply add up the corresponding constants $c_{\rm LG}'$).

Suppose first that $a$ and $\nu$ are zero, and let
\[
\nu^n(x,d\zeta) =  \frac{1}{\varepsilon}\delta_{\varepsilon b(x)}(d\zeta) \bm1_{\zeta\ne0},
\]
where $\varepsilon=n^{-1}$. Clearly \ref{L_approxA_1} holds. Moreover, $A^nf(x,z)=\varepsilon^{-1} (f(z+\varepsilon b(x))-f(z))$ lies in $C(\R^d\times\R^k)$, and converges to $Af(x,z)=b(x)^\top\nabla f(z)$. The convergence is locally uniform, since the difference quotients converge locally uniformly for $f\in C^2(\R^k)$. Thus \ref{L_approxA_3} holds. Finally, note that $b^n(x)=b(x)$, and that $\int_{\R^k}|\zeta|^q\nu^n(x,d\zeta)=\varepsilon^{q-1}|b(x)|^q$ for any $q\ge2$. Thus it follows from \eqref{eq_LG} that \eqref{L_approxA_2_LG} holds with $c_{\rm LG}'=3c_{\rm LG}$.

Suppose instead that $b$ and $\nu$ are zero. Write $\sigma(x)=a(x)^{1/2}$ using the positive semidefinite square root. Then $x\mapsto\sigma(x)$ is again continuous and compactly supported. So are its columns, denoted by $\sigma_1(x),\ldots,\sigma_d(x)$. Let
\[
\nu^n(x,d\zeta) =  \frac{1}{2\varepsilon^2} \sum_{i=1}^d (\delta_{\varepsilon\sigma_i(x)}(d\zeta)+\delta_{-\varepsilon\sigma_i(x)}(d\zeta))\bm1_{\zeta\ne0},
\]
where again $\varepsilon=n^{-1}$. As before, \ref{L_approxA_1} holds. Moreover,
\begin{align*}
A^n f(x,z) &= \frac12 \sum_{i=1}^d \frac{f(z+\varepsilon\sigma_i(x)) - 2f(z) + f(z-\varepsilon\sigma_i(x))}{\varepsilon^2} \\
&\to \frac12 \sum_{i=1}^d \sigma_i(x)^\top \nabla^2 f(z) \sigma_i(x) =  \frac12\tr(a(x)\nabla^2f(z)).
\end{align*}
Again, $A^nf$ lies in $C(\R^d\times\R^k)$ and the convergence is locally uniform since $f$ is $C^2$ and the $\sigma_i$ are continuous. This gives \ref{L_approxA_3}. Next, we have $b^n(x)=0$. Also, writing $\sigma_i^j(x)$ for the $j$th component of $\sigma_i(x)$, we have
\[
\int_{\R^k}|\zeta|^q\nu^n(x,d\zeta)=\varepsilon^{q-2}\sum_{i=1}^d |\sigma_i(x)|^q \le \Big( \sum_{i,j=1}^d |\sigma_i^j(x)|^2\Big)^{q/2} = \tr(a(x))^{q/2}
\]
for any $q\ge2$. Since also $\tr(a(x))\le\sqrt{d}\,|a(x)|$, it follows from \eqref{eq_LG} that \eqref{L_approxA_2_LG} holds with $c_{\rm LG}'=2\sqrt{d}\,c_{\rm LG}$.

Finally, suppose that $b$ and $a$ are zero. Let $\varphi_n$ be a continuous cutoff function supported on $[n^{-1},n]$ and equal to one on $[2n^{-1},n/2]$. We arrange so that $\varphi_{n+1}\ge\varphi_n$ for all $n$. Let
\[
\nu^n(x,B) = \int_{\R^k} \left(\delta_\zeta(B) + \frac1\varepsilon\delta_{-\varepsilon\zeta}(B)\right) \varphi_n(|\zeta|) \nu(x,d\zeta),
\]
where again $\varepsilon=n^{-1}$. Clearly $\nu^n(x,\fdot)$ has compact support. Moreover,
\begin{align*}
\nu^n(x,\R^k) &\le \left(1+\frac1\varepsilon\right) \int_{\R^k} n^2 |\zeta|^2 \nu(x,d\zeta) \\
&\le c_{\rm LG}(1+n)n^2  \sup_{x\in Q}(1+|x|^2) < \infty,
\end{align*}
due to the growth bound \eqref{eq_LG} and recalling that we assumed $\nu(x,d\zeta)=0$ for all $x$ outside some compact set $Q$. We deduce that \ref{L_approxA_1} holds. Next, we have
\[
b^n(x) = \int_{\R^k} \left(\zeta + \frac1\varepsilon (-\varepsilon \zeta)\right) \varphi_n(|\zeta|) \nu(x,d\zeta) = 0
\]
and
\[
\int_{\R^k}|\zeta|^q\nu^n(x,d\zeta) = 2 \int_{\R^k} |\zeta|^q \varphi_n(|\zeta|) \nu(x,d\zeta) \le 2 \int_{\R^k} |\zeta|^q \nu(x,d\zeta).
\]
Thus it follows from \eqref{eq_LG} that \eqref{L_approxA_2_LG} holds with $c_{\rm LG}'=2c_{\rm LG}$.
It remains to show that $A^nf\to Af$ locally uniformly. Write
\begin{align*}
Af(x,z)&-A^nf(x,z)\\
&= \int_{\R^k}\left(f(z+\zeta)-f(z)-\zeta^\top\nabla f(z)\right)(1-\varphi_n(|\zeta|))\nu(x,d\zeta) \\
&\quad + \int_{\R^k} \frac{1}{\varepsilon} \left(f(z) - f(z-\varepsilon\zeta) - \varepsilon\zeta^\top\nabla f(z)\right) \varphi_n(|\zeta|)\nu(x,d\zeta).
\end{align*}
Due to \eqref{eq_MGP_pf_existence_3} and the bound
\[
|f(z) - f(z-\varepsilon\zeta) - \varepsilon\zeta^\top\nabla f(z)| \le \frac{\varepsilon^2}{2} \|\nabla^2f\|_\infty |\zeta|^2,
\]
we obtain
\begin{equation}\label{eq_L_approx_pf_6_new}
\begin{aligned}
|Af(x,z)-A^nf(x,z)| \le c \int_{\R^k}  (1-\varphi_n(|\zeta|)) \widetilde\nu(x,d\zeta) + \frac{c}{n} \widetilde\nu(x,\R^k) 
\end{aligned}
\end{equation}
for the constant $c=\frac12\|\nabla^2f\|_\infty$ and the finite kernel
\[
\widetilde\nu(x,d\zeta) = |\zeta|^2 \nu(x,d\zeta).
\]
Thanks to the growth bound \eqref{eq_LG} and the assumption that $\nu(x,d\zeta)=0$ for all $x$ outside a compact set $Q$, we have $\int_{\R^k}  |\zeta|^2 \nu(x,d\zeta)\le c_{\rm LG}\sup_{x\in Q}(1+|x|^2)<\infty$. Thus the second term on the right-hand side of \eqref{eq_L_approx_pf_6_new} tends to zero uniformly as $n\to\infty$. To bound the first term, write
\begin{equation}\label{eq_L_approx_pf_4_new}
\int_{\R^d}   (1-\varphi_n(|\zeta|)) \widetilde\nu(x,d\zeta) \le \int_{\R^d}  \psi_n(|\zeta|) \widetilde\nu(x,d\zeta) + \int_{\R^d}  \bm1_{|\zeta|\ge n/2} \widetilde\nu(x,d\zeta),
\end{equation}
where $\psi_n=(1-\varphi_n)\bm1_{[0,2n^{-1}]}$ is continuous and supported on $[0,2n^{-1}]$. We bound the two terms on the right-hand side of \eqref{eq_L_approx_pf_4_new} separately.

First, by assumption, $x\mapsto\widetilde\nu(x,d\zeta)$ is continuous from $\R^d$ to $M_+(\R^k)$. Moreover, $\widetilde\nu(x,d\zeta)$ is zero for $x$ outside a compact set $Q$. Thus the set $P=\{\widetilde\nu(x,d\zeta)\colon x\in\R^d\} = \{\widetilde\nu(x,d\zeta)\colon x\in Q\}$ is a compact subset of $M_+(\R^d)$, being a continuous image of a compact set. Therefore $P$ is tight, so that
\begin{equation}\label{eq_L_approx_pf_2_new}
\sup_{x\in\R^d} \int_{\R^d} \bm1_{|\zeta|\ge n/2} \widetilde\nu(x,d\zeta) = \sup_{\mu\in P} \mu( B(0,n/2)^c ) \to 0, \quad n\to\infty.
\end{equation}
Next, we claim that
\begin{equation}\label{eq_L_approx_pf_1_new}
\limsup_{n\to\infty} \sup_{x\in\R^d} \int_{\R^d}  \psi_n(|\zeta|) \widetilde\nu(x,d\zeta) = 0.
\end{equation}
Let $v$ denote the limsup in \eqref{eq_L_approx_pf_1_new}. For each $n$, $x\mapsto\int_{\R^d}  \psi_n(|\zeta|) \widetilde\nu(x,d\zeta)$ is continuous and supported on $Q$, hence maximized at some $x_n\in Q$. After passing to a subsequence, we have $x_n\to\bar x$ for some $\bar x\in Q$, and $\int_{\R^d}  \psi_n(|\zeta|) \widetilde\nu(x_n,d\zeta)\to v$. By the choice of $\varphi_n$, we have $\psi_{n+1}\le\psi_n$ for all $n$. As a result, for each fixed $m$,
\[
v\le\lim_{n\to\infty}\int_{\R^d}  \psi_m(|\zeta|) \widetilde\nu(x_n,d\zeta)=\int_{\R^d}  \psi_m(|\zeta|) \widetilde\nu(\bar x,d\zeta).
\]
This tends to zero as $m\to\infty$ by dominated convergence, since $\widetilde\nu(\bar x,\{0\})=0$. Thus $v=0$, that is, \eqref{eq_L_approx_pf_1_new} holds. Combining \eqref{eq_L_approx_pf_4_new}, \eqref{eq_L_approx_pf_2_new}, and \eqref{eq_L_approx_pf_1_new}, it follows that also the first term on the right-hand side of \eqref{eq_L_approx_pf_6_new} tends to zero uniformly as $n\to\infty$. This gives \ref{L_approxA_3} and completes the proof of the lemma.
\end{proof}

We can now complete the proof of existence of weak $L^p$ solutions.

\begin{proof}[Proof of Theorem~\ref{T_weak_existence}]
Consider the kernels $\nu^n(x,d\zeta)$ and corresponding triplets $(b^n,0,\nu^n)$ given by Lemma~\ref{L_approxA}. Apply the basic existence result Lemma~\ref{L_basic_existence} with each kernel $\nu^n(x,d\zeta)$ and the given $g_0$ and $K$ to obtain processes $(X^n,Z^n)$. Note that the differential characteristics of $Z^n$ with respect to the ``truncation function'' $\chi(\zeta)=\zeta$ are $b^n(X^n),a^n(X^n)=0,\nu^n(X^n,d\zeta)$. Thus $(X^n,Z^n)$ is a weak $L^p$ solution of \eqref{eq_SVEC} for the data $(g_0,K,b^n,0,\nu^n)$.

The triplets $(b^n,0,\nu^n)$ satisfy the growth bound in Lemma~\ref{L_approxA}\ref{L_approxA_2} with a common constant $c_{\rm LG}'$. Corollary~\ref{C_tightness} thus implies that the sequence $\{X^n\}_{n\in\N}$ is tight in $L^p_{\rm loc}$. By passing to a subsequence, we assume that $X^n\Rightarrow X$ in $L^p_{\rm loc}$ for some limiting process $X$.

We claim that the sequence $\{Z^n\}_{n\in\N}$ is tight in $D$. To prove this, first note that for any $T\in\R_+$, $m>0$, $\varepsilon>0$, we have
\begin{align*}
\P\left( \int_0^T \int_{\R^k} \bm1_{|\zeta|>m} \nu^n(X^n_t,d\zeta) > \varepsilon \right) &\le \frac{1}{m^2\varepsilon} \E\left[ \int_0^T \int_{\R^k} |\zeta|^2 \nu^n(X^n_t,d\zeta) \right] \\
&\le \frac{1}{m^2\varepsilon} c_{\rm LG}'  \left( T + \E[ \|X^n\|_{L^2(0,T)}^2 ] \right).
\end{align*}
Theorem~\ref{T_apriori} shows that the expectation on the right-hand side is bounded by a constant that does not depend on $n$. Therefore,
\[
\lim_{m\to\infty}\sup_{n\in\N} \P\left( \int_0^T \int_{\R^k} \bm1_{|\zeta|>m} \nu^n(X^n_t,d\zeta) > \varepsilon \right) = 0.
\]
Furthermore, the increasing process
\begin{equation}\label{eq_proof_weak_existence_1}
\int_0^t \left( |b^n(X^n_s)| + \int_{\R^k} |\zeta|^2 \nu^n(X^n_s,d\zeta) \right) ds, \quad t\ge0,
\end{equation}
is {\em strongly majorized} by $c_{\rm LG}' \int_0^\fdot (1 + |X^n_s|^2)ds$ in the sense that the difference of the two is increasing; see \citet[Definition~VI.3.34]{jac_shi_03}. The latter process converges weakly to the continuous increasing process $c_{\rm LG}' \int_0^\fdot (1 + |X_s|^2)ds$. Thus \eqref{eq_proof_weak_existence_1} is tight with only continuous limit points; see \citet[Proposition~VI.3.35]{jac_shi_03}. With these observations we may now apply \citet[Theorem~VI.4.18 and Remark~VI.4.20(2)]{jac_shi_03} to conclude that $\{Z^n\}_{n\in\N}$ is tight in $D$.

Finally, by passing to a further subsequence, we now have $(X^n,Z^n)\Rightarrow(X,Z)$ in $L^p_{\rm loc}\times D$ for some limiting process $(X,Z)$. An application of Theorem~\ref{T_stability} then shows that $(X,Z)$ is a weak $L^p$ solution of \eqref{eq_SVEC} for the data $(g_0,K,b,a,\nu)$, as desired. The proof of Theorem~\ref{T_weak_existence} is complete.
\end{proof}

\section{Uniqueness of weak $L^p$ solutions}
\label{s:uniqueness}

We now turn to pathwise uniqueness and uniqueness in law under suitable Lipschitz conditions.

Let $(X,Z)$ be a weak $L^p$ solution of \eqref{eq_SVEC} for the data $(g_0,K,b,a,\nu)$, where $\int_{\mathbb{R}^k} |\zeta|^2\nu(x,d\zeta) < \infty$. The characteristics are understood with respect to the ``truncation function'' $\chi(\zeta)=\zeta$. Standard representation theorems for semimartingales allow us to express $Z$ as a stochastic integral with respect to time, Brownian motion, and a compensated Poisson random measure; see \citet[Theorem 2.1.2]{JP11} and \citet{EKL77,LM76}. It follows that $X$ satisfies a $d$-dimensional stochastic Volterra equation of the form
\begin{equation}\label{eq:jumpsveLip}
\begin{split}
X_t &= g_0(t)+\int_0^t K(t-s)  b(X_s)ds + \int_0^t K(t-s) \sigma(X_s)dW_s \\ &\quad + \int_{[0,t) \times \R^m}   K(t-s) \gamma(X_{s},\xi) (\mu(ds,d\xi)-F(d\xi)ds), \quad \text{$\P\otimes dt$-a.e.}
\end{split}
\end{equation}
for some $d'$-dimensional Brownian motion $W$, Poisson random measure $\mu$ on $\R_+ \times \R^m$ with compensator $dt\otimes F(d\xi)$, and some measurable functions $\sigma\colon\R^d \to \R^{k\times d'}$ and $\gamma\colon\R^d \times \R^{m} \to \R^k$ such that
\[
\text{$a(x)= \sigma(x) \sigma(x)^{\top}$ \quad and \quad $\nu(x,B) = \int_{\R^m} \bm1_B(\gamma(x,\xi))F(d\xi)$.}
\]
Both $W$ and $\mu$ are defined on some extension $(\Omega,\Fcal,\F,\P)$ of the filtered probability space where $X$ and $Z$ are defined.

Conversely, given $(g_0,K,b,\sigma,\gamma,F)$ along with a filtered probability space $(\Omega,\Fcal,\F,\P)$ equipped with a $d'$-dimensional Brownian motion $W$ and Poisson random measure $\mu$ on $\R_+ \times \R^m$ with compensator $dt\otimes F(d\xi)$, a \emph{solution of \eqref{eq:jumpsveLip}} is any predictable process $X$ on $(\Omega,\Fcal,\F,\P)$ with trajectories in $L^p_{\rm loc}$ such that \eqref{eq:jumpsveLip} holds. We are now in position to define pathwise uniqueness for such solutions. 

\begin{definition}\label{def:pathwiseuniqueness}
Fix $(g_0,K,b,\sigma,\gamma,F)$ as above. We say that \emph{pathwise uniqueness} holds for \eqref{eq:jumpsveLip} if for any $(\Omega,\Fcal,\F,\P)$, $W$, $\mu$ as above and any two solutions $X$ and $Y$ of \eqref{eq:jumpsveLip}, we have $X=Y$, $\mathbb{P} \otimes dt$-a.e.
\end{definition}

The powerful abstract machinery of \citet{K:14} can be used in this setting to relate pathwise uniqueness and weak existence  to strong existence and uniqueness in law. A \emph{strong solution} of \eqref{eq:jumpsveLip} in the sense of \citet[Definition 1.2]{K:14} is a weak $L^p$ solution $X$ which is $\mathbb{P} \otimes dt$-a.e.\ equal to a Borel measurable function of $W$ and $N=\int_0^{\cdot} \xi (\mu(d\xi,ds)-F(d\xi) ds)$ from \eqref{eq:jumpsveLip}.

\begin{theorem}\label{T_Kurtz_0}
The following are equivalent:
\begin{enumerate}
\item There exists a weak $L^p$ solution of \eqref{eq_SVEC}, and pathwise uniqueness holds for \eqref{eq:jumpsveLip}.
\item There exists a strong solution of \eqref{eq:jumpsveLip}, and joint uniqueness in law of $(X,W,N)$ holds.
\end{enumerate}
\end{theorem}

\begin{proof}
Let $S_1= L^p_{\rm{loc}} $ and $S_2= D \times D$. 
Then the statement follows from \citet[Theorem~1.5 and Lemma~2.10]{K:14}. Indeed, \citet[Lemma 2.10]{K:14} clarifies that our notion of pathwise uniqueness is equivalent to the one used in \citet[Theorem 1.5]{K:14}. Note that the definitions in \citet[Definition 1.4, Definition 2.9]{K:14} have to be adapted to replace $\mathbb{P}$-a.s.\ assertions by $\mathbb{P} \otimes dt$-a.e.\ assertions. 
\end{proof}

As for standard SDEs, pathwise uniqueness holds under Lipschitz conditions on the coefficients.

\begin{theorem}\label{T_uniqueness}
Let $K \in L^2_{\rm loc}$ and suppose there exists a constant $c_{\rm Lip}$ such that $b,\sigma,\gamma,F$ in \eqref{eq:jumpsveLip} satisfy
\[
\begin{aligned}
	|b(x)&-b(y)|^2 +  |\sigma(x)-\sigma(y)|^2 \\&+ \int_{\R^m} | \gamma(x,\xi) - \gamma(y,\xi)|^2 F(d\xi)  \leq c_{\rm Lip} | x-y|^2
\end{aligned}
\]
for all $x,y \in \R^d$. Then pathwise uniqueness holds for \eqref{eq:jumpsveLip}, and hence also uniqueness in law of weak $L^p$ solutions of \eqref{eq_SVEC}.
\end{theorem}

\begin{proof}
The argument is similar to the proof of \eqref{eq:estimateLp}, so we only give a sketch.
Let $X$ and $Y$ be two solutions of \eqref{eq:jumpsveLip} with trajectories in $L^2_{\rm loc}$. Define $\tau_n =\inf\{t\colon \int_0^t (|X_s|^2 + |Y_s|^2) ds\geq n\}\wedge T$ as well as $X^n_t=X_t\bm1_{t<\tau_n}$ and  $Y^n_t= Y_t\bm1_{t<\tau_n}$. As in the proof of \eqref{eq:estimateLp}, but relying on the Lipschitz assumption rather than linear growth, one shows that
\[
\E[\|X^n-Y^n\|_{L^2(0,T)}^2]\leq c\int_0^T\int_0^t |K(t-s)|^2\E[|X^n_s-Y^n_s|^2])ds\,dt
\]
for all $T\ge0$ and some $c=c(T,c_{\rm Lip})<\infty$ that depends continuously on $T$ and $c_{\rm Lip}$. Multiple changes of variables and applications of Tonelli's theorem then show that $f_n(t)= \E[\|X^n-Y^n\|_{L^2(0,t)}^2]$ satisfies the convolution inequality $f_n(t) \le  - (\widehat K * f_n)(t)$ on $[0,T]$ with $\widehat K=- c(T,c_{\rm Lip})|K|^2$. The Gronwall lemma for convolution inequalities (see Lemma~\ref{L_Gronwall}) yields $f_n(T)\le0$, and monotone convergence gives $f_n(T)\to \E[ \|X-Y\|_{L^2(0,T)}^2]$. Thus $\E[ \|X-Y\|_{L^2(0,T)}^2 ]=0$, which implies pathwises uniqueness in the sense of Definition~\ref{def:pathwiseuniqueness}. Uniqueness in law now follows from Theorem~\ref{T_Kurtz_0}.
\end{proof}

\section{Path regularity}\label{s:pathreg}

Solutions $X$ of \eqref{eq_SVEC} can be very irregular. Consider for example the simple case
\[
X_t = \int_{[0,t)} K(t-s)dN_s= \sum_{t >T_n} K(t-T_n),
\]
where $N$ is a standard Poisson process with jump times $T_n$, $n\in\N$. Without further information about $K$, nothing can be said about the path regularity of $X$ beyond measurability. Even with singular but otherwise ``nice'' kernels such as those in Example~\ref{ex:kernels}\ref{ex:kernels_1}, $X$ fails to have c\`adl\`ag or even l\`adl\`ag trajectories. This is why $L^p$ spaces are useful for the solution theory. Nonetheless, one frequently does have additional information that implies better path regularity.

The following result yields H\"older continuity in many cases, also when the driving semimartingale has jumps. The result relies on a combination of the estimates  \eqref{eq:estimateLp}-\eqref{eq:estimateW} with Sobolev embedding theorems. For any $T>0$ and $\eta>0$, we denote by $C^{\eta}(0,T)$ the space of H\"older continuous functions of order $\eta$ on  $[0,T]$. Thus $f \in C^{\eta}(0,T)$ if 
\[
\| f\|_{C^{\eta}(0,T)}= \| f\|_{L^{\infty}(0,T)} + \sup_{\substack{t,s \,\in \, [0,T] \\ t \neq s} } \frac{|f(t)-f(s)|}{|t-s|^{\eta}} < \infty.
\]

\begin{theorem}\label{T_pathreg}
Let $d,k\in\N$, $p\in[2,\infty)$, and consider data $(g_0,K,b,a,\nu)$ as in \ref{SVE_IC}--\ref{SVE_abnu}. Assume there exist
a constant $\eta\in(0,1)$, a locally bounded function $c_K\colon\R_+\to\R_+$,
and a constant $c_{\rm LG}$ such that \eqref{eq_K_Wpnu} and \eqref{eq_LG} hold. Then for any weak $L^p$ solution $X$ of \eqref{eq_SVEC} the following statements hold:
	\begin{enumerate}
		\item \label{T:reg1}
		if $\eta p>1$, then $X-g_0$ admits a version whose sample paths lie in $C^{(\eta p -1)/p}(0,T)$ almost surely.
		\item \label{T:reg2}
				if $p=2$ and $\nu\equiv 0$, then $X-g_0$ admits a version   whose sample paths lie in $C^{\beta}(0,T)$ for all $\beta < \eta$ almost surely.
\item \label{T:reg3} if $K(0)<\infty$ and if $K-K(0)$ (instead of $K$) satisfies \eqref{eq_K_Wpnu} with $\eta p>1$, then $X-g_0$ admits a version with  c\`agl\`ad sample paths.
		\item \label{T:reg4} without assuming \eqref{eq_K_Wpnu} and \eqref{eq_LG}, but rather that $K$ is differentiable with derivative $K' \in L^2_{\rm loc}$, we have that $X-g_0$ is a semimartingale and thus admits a version with c\`agl\`ad sample paths.\footnote{Note that \eqref{eq_K_Wpnu} is implied by the given assumption on $K$, for any $\eta<1/p$.}
		\end{enumerate}
		\end{theorem}
		
\begin{proof} Assertion \ref{T:reg1} follows from \eqref{eq:estimateW} and the Sobolev embedding theorem, see \citet[Theorem 8.2]{di2012hitchhiker}. To prove \ref{T:reg2}, one can adapt the proof of Theorem~\ref{T_apriori} to get that \eqref{eq:estimateLp}-\eqref{eq:estimateW} hold for all {$p \geq 2$}. Applying \citet[Theorem 8.2]{di2012hitchhiker} for sufficiently large values of $p$ yields the claimed statement. For \ref{T:reg3}, we write 
$$ X_t-g_0(t)=K(0)Z_{t-} + \int_{[0,t)}(K(t-s)-K(0))dZ_s.$$
	The claimed regularity follows on observing that the first term on the right-hand side is c\`agl\`ad and that, similarly to \ref{T:reg1}, the second term admits a version with continuous sample paths.  
For \ref{T:reg4} one applies a Fubini theorem, see Lemma \ref{L:Fubini}, to get that
	\[
	X_t-g_0(t)=K(0)Z_{t-} + \int_0^t \Big(\int_{[0,s)}K'(s-u)dZ_u\Big)ds.
	\]
This completes the proof.
\end{proof}

\section{Applications}\label{S:applications}
In this section, we illustrate our results with two applications: scaling limits of Hawkes processes and approximation of stochastic Volterra equations by Markovian semimartingales.

\subsection{Generalized nonlinear Hawkes processes and their scaling limits}\label{S_app_1}

Fix $d,k\in \N$ along with functions $g_0\colon\R_+ \to \R^d$, $b\colon\R^d \to \R^k$, $\Lambda\colon \R^d \to \R^k_+$, and a kernel $K\colon\R_+ \to \R^{d\times k}$. We fix $p\ge2$ and assume that $g_0$ and $K$ lie in $L^p_{\rm loc}$, that $K$ satisfies \eqref{eq_K_Wpnu} for some $\eta\in(0,1)$ and locally bounded function $c_K$, and that $b$ and $\Lambda$ are continuous and satisfy the linear growth condition
\begin{align}\label{eq:lambdagrowth}
|b(y)|+ |\Lambda(y)|\leq c(1+|y|), \quad y \in \R^d,
\end{align}
for some constant $c\in\R_+$.
Consider a $k$-dimensional counting process $N$ with no simultaneous jumps, whose intensity vector is given by $\Lambda(Y)$ with $Y$ a $d$-dimensional predictable process with trajectories in $L^p_{\rm loc}$ that satisfies
\begin{align}
Y_t &= g_0(t) + \int_0^t K(t-s)b(Y_s)ds  + \int_{[0,t)}  K(t-s) dN_s \quad \text{$\P\otimes dt$-a.e.} \label{eq:Hawkesnonlinear1} 
\end{align}
We call such a process $N$ a \emph{generalized nonlinear Hawkes process}. The existence of $Y$ and $N$ follows immediately from Theorem~\ref{T_weak_existence}. Indeed, \eqref{eq:Hawkesnonlinear1} is a stochastic Volterra equation of the form \eqref{eq_SVEC} whose driving semimartingale $Z$ has differential characteristics $b(Y)$, $a(Y)=0$, and $\nu(Y,d\zeta)=\sum_{i=1}^d \Lambda_i(Y) \delta_{e_i}(d\zeta)$, where $e_1,\ldots,e_d$ are the canonical basis vectors in $\R^d$.

\begin{example}
For $k=d$ and $b=0$, we obtain nonlinear multivariate Hawkes processes in the spirit of \citet{bremaud1996stability,daley2003introduction,delattre2016hawkes} and the references there.
\end{example}

We now establish convergence of rescaled generalized nonlinear Hawkes processes toward stochastic Volterra equations with no jump part, as those studied by \citet{AJLP17}. In the following theorem we consider given inputs $g_0,K$ as well as $g_0^n,K^n$ indexed by $n\in\N$, that satisfy the assumptions described in the beginning of this subsection. We consider a fixed function $\Lambda=(\Lambda_1,\ldots,\Lambda_d)$ as above and take $b=-\Lambda$. We continue to assume \eqref{eq:lambdagrowth} (with $b=-\Lambda$). For each $n$, denote the corresponding generalized nonlinear Hawkes process by $N^n$. Its intensity vector is $\Lambda(Y^n)$, where $Y^n$ satisfies
\begin{align*}
Y^n_t &= g_0^n(t)+ \int_{[0,t)} K^n(t-s)dM^n_s, \\
M^n_t &= N_t^n - \int_0^t \Lambda(Y^n_s) ds. 
\end{align*}

\begin{theorem}\label{T:Hawkes}
For each $n \in \N$, consider a diagonal matrix of rescaling parameters, $\varepsilon^n=\diag(\varepsilon^n_1,\ldots,\varepsilon^n_d) \in \R^{d\times d}$. Assume for all $i$ that 
\begin{align}\label{eq:scalinggrowth}
n(\varepsilon^n_i)^{2}  \Lambda_i\left((\varepsilon^n)^{-1}x\right) \leq c_i (1 + |x|^2),\quad x \in \R^d,
\end{align}
for some constant $c_i>0$ independent of $n$, and that
\begin{align}\label{eq:scaling}
	n (\varepsilon^n_{i})^2 \Lambda_i\left((\varepsilon^n)^{-1} x\right)   \to \bar \Lambda_i(x)
	\end{align}
	locally uniformly in $x$ for some function $\bar \Lambda\colon\R^d \to \R^d$. 
	Assume also that 
	\begin{enumerate}
	\item \label{eq:Hawkes0}  $\varepsilon^n g_0^n(n \fdot)\to g_0$ in $L^p_{\rm loc}$,
		\item \label{eq:Hawkesi}  $\varepsilon^nK^n(n\fdot)(\varepsilon^n)^{-1}\to K$ in $L^p_{\rm loc}$,
		\item \label{eq:Hawkesii} $\varepsilon^nK^n(n\fdot)(\varepsilon^n)^{-1}$ satisfy \eqref{eq_K_Wpnu} with the same $\eta$ and $c_K$ as $K$.
	\end{enumerate}
	Then the rescaled sequence $(X^n,Z^n)$ given by $X^n_t=\varepsilon^n Y^n_{nt}$, $Z^n_t=\varepsilon^n M^n_{nt}$ is tight in $L^p_{\rm loc} \times D$, and every limit point $(X,Z)$ is a weak $L^p$ solution of
\begin{equation}\label{eq:cir1}
X_t = g_0(t)  +  \int_0^t K(t-s){dZ_s},
\end{equation}
where $Z$ admits the representation $Z_t = \int_0^t \sqrt{  \diag(\bar \Lambda(X_s))}dW_s$ for some $d$-dimensional Brownian motion $W$.
\end{theorem}

\begin{proof}
One verifies that the rescaled intensity $X^n$ satisfies the equation
\[
X^n_t = {\varepsilon^n} g^n_0(nt) + \int_0^t \varepsilon^n K^n(n(t-s))(\varepsilon^n)^{-1} dZ^n_s,
\]
where $Z^n$ has differential characteristics $b^n(X^n)=0, a^n(X^n)=0, \nu^n(X^n,d\zeta)$ with jump kernel given by $\nu^n(x,d\zeta) =  \sum_{i=1}^d n \Lambda_i\left((\varepsilon^n)^{-1}x\right) \delta_{ \varepsilon^n_i e_i}(d\zeta)$. Here $e_1,\ldots,e_d$ are the canonical basis vectors in $\R^d$. The  associated operator is given by
\[
A^n f(x,z) =  \sum_{i=1}^d n \Lambda_i\left((\varepsilon^n)^{-1}x\right) \left( f(z+ \varepsilon^n_i e_i) - f(z) - \varepsilon^n_i \nabla f(z)^\top e_i\right),
\]
which converges locally uniformly to $\frac 1 2 \tr\left(  \diag\left(\bar \Lambda(x)\right) \nabla^2  f(z)\right)$ due to \eqref{eq:scaling}. Consequently, provided $(X^n,Z^n)$ is tight, Theorem~\ref{T_stability_MGP} shows that every limit point $(X,Z)$ is a weak $L^p$ solution of \eqref{eq:cir1}, where $Z$ has differential characteristics $b(X)=0$, $a(X)=\diag\left(\bar \Lambda(X)\right)$, $\nu(X,d\zeta)=0$. The representation of $Z$ in terms of a Brownian motion is standard. It remains to prove tightness. First, by virtue of \eqref{eq:scalinggrowth}, we have $\int_{\R^d}|\zeta|^2\nu^n(x,d\zeta) \leq c (1 + |x|^2)$ for all $x\in\R^d$ and some constant $c$.
	Thus, \eqref{eq_LG} is satisfied uniformly in $n$. Recalling \ref{eq:Hawkes0} and \ref{eq:Hawkesii}, Corollary~\ref{C_tightness} yields tightness of $(X^n)_{n \geq 1}$. Tightness of $(Z^n)_{n \geq 1}$ in $D$ is then obtained by reiterating the arguments in the proof of Theorem~\ref{T_weak_existence} at the end of Section~\ref{S_pf_Tex1}.  Since marginal tightness implies joint tightness the proof is complete.
\end{proof}

\begin{example}
Let $K,g_0$ be as described in the beginning of this subsection and let $\varepsilon^n=\diag(\varepsilon^n_1,\ldots,\varepsilon^n_d) \in \R^{d\times d}$ as above. Then the functions $g_0^n$ and $K^n$ given by 
\[
g_0^n(t)=(\varepsilon^n)^{-1}g_0\left(\frac t n\right), \qquad K^n(t)=(\varepsilon^n)^{-1} K\left(\frac t n\right) \varepsilon^n
\]
satisfy \ref{eq:Hawkes0}--\ref{eq:Hawkesii}. There are other ways of constructing such kernels, as illustrated in \citet{JR:15,JR:16} for linear Hawkes processes. 
\end{example}

Theorem \ref{T:Hawkes} is in the same spirit as the results of \citet{erny2019mean}, who obtain square-root type processes as limits of mean field interactions of multi-dimensional nonlinear Hawkes processes. The following example provides a concrete specification for the special case of fractional powers, extending results in \citet{JR:15,JR:16}  to nonlinear Hawkes processes. 

\begin{example}
Let $\beta_i \in (0,2)$, $i=1,\ldots,d$, and take $\Lambda(y)=(y_1^{\beta_1},\ldots,y_d^{\beta_d})$. Let $\varepsilon^n= \diag(\varepsilon^n_1,\ldots,\varepsilon^n_d)$ satisfy $n (\varepsilon^n_i)^{2-\beta_i} \to \nu_i$ for some constants $\nu_i \geq 0$. Then \eqref{eq:scalinggrowth}--\eqref{eq:scaling} are satisfied with $\bar \Lambda=\Lambda$. The limiting process $(X,Z)$ produced by Theorem~\ref{T:Hawkes} takes the form
\[
X_t = g_0(t)  +  \int_0^t K(t-s)\sqrt{  \diag\left(\nu_1|X^1_s|^{\beta_1},\ldots,\nu_d|X^d_s|^{\beta_d}\right)}dW_s,
\]
where $W$ is a $d$-dimensional Brownian motion. 
\end{example}
 	
We end this subsection with some comments regarding the integrability conditions on the kernel. Our work aims to develop a theory of stochastic Volterra equations with continuous as well as discontinuous trajectories. Having this goal in mind, the $L^2$ integrability condition on the kernel is used to define stochastic integrals with respect to the continuous martingale part and the discontinuous martingale part with non-summable jumps of the driving semimartingale $Z$ in \eqref{eq_SVEC}. In some particular instances, however, it is possible to weaken the $L^2$ integrability condition on the kernel. For example, Lemma \ref{L_basic_existence} yields existence of solutions with bounded jump intensity assuming only measurability of the kernel. This can be applied to \eqref{eq:Hawkesnonlinear1} when $\Lambda$ is bounded. If $\Lambda$ is not bounded then $L^1$ integrability conditions are sufficient to prove the existence of Hawkes processes, see for instance \citet[Theorem 1]{bremaud1996stability}. When the driving semimartingale $Z$ has affine characteristics, kernels that are locally in $L^1$ can also be considered by studying an ``integrated version" of \eqref{eq_SVEC} in the spirit of \eqref{D_MP_1}. This approach is taken in \citet{aj19weak} to obtain existence, uniqueness and stability results in a framework including $L^1$ kernels as well as continuous and infinite activity jump processes. In this case the characteristics of $Z$ are no longer necessarily absolutely continuous with respect to the Lebesgue measure.

\subsection{Approximation by Markovian semimartingales}\label{S_app_2}
It is sometimes useful, for example for numerical purposes, to replace a singular kernel with a smooth approximation. Theorem~\ref{T_stability_MGP} can be used to analyze this procedure; see also the stability result of \citet[Theorem~3.6]{AJEE:19a} for the case without jumps. An approximation scheme that is useful in practice is to consider weighted sums of exponentials.

\begin{theorem}\label{T:factors}
Fix $d,k\in \N$, $p \geq 2$ and $(g_0,K,b,a,\nu)$ as in \ref{SVE_IC}--\ref{SVE_abnu}, and assume \eqref{eq_LG} holds. For each $n \in \N$,  let $g_0^n \in L^p_{\rm loc}$ and consider the kernel
\[
K^n(t) = \sum_{i=1}^n c^n_i e^{-\lambda^n_i t }
\]
for some $c^n_i \in \R^{d \times k}$ and $\lambda^n_i \geq 0$, $i=1,\ldots,n$. By Example~\ref{ex:kernels}\ref{ex:kernels:2} and Theorem \ref{T_weak_existence} there exists a weak $L^p$-solution $(X^n,Z^n)$ for the data $(g_0^n,K^n,b,a,\nu)$.  Moreover, $X^n$ admits the representation
\begin{align*}
X^n_t & = g^n_0(t) + \sum_{i=1}^n c^n_i Y^{n,i}_t\\
dY^{n,i}_t &= -\lambda^n_i Y^{n,i}_tdt + dZ^n_t, \quad Y^{n,i}_0=0, \quad i=1,\ldots,n.
\end{align*}
Assume in addition that
\begin{enumerate}
\item \label{eq:Hawkess0}  $g_0^n\to g_0$ in $L^p_{\rm loc}$,
\item \label{eq:Hawkessi}  $K^n\to K$ in $L^p_{\rm loc}$,
\item \label{eq:Hawkessii}  $K^n$ satisfy \eqref{eq_K_Wpnu} with the same $\eta$ and $c_K$ as $K$.%
\end{enumerate}
Then $(X^n,Z^n)_{n \geq 1}$ is tight in $L^p_{\rm loc}\times D$, and every limit point $(X,Z)$ is a weak $L^p$ solution of \eqref{eq_SVEC} for the data $(g_0,K,b,a,\nu)$.
\end{theorem}

\begin{proof}
Defining $Y^{n,i}_t=\int_0^t e^{-\lambda^n_i (t-s)}dZ^n_s$, the representation of $X^n$ follows from It\^o's formula. Corollary~\ref{C_tightness} yields tightness of $(X^n)_{n \geq 1}$. Tightness of $(Z^n)_{n \geq 1}$ in $D$ is then obtained by reiterating the arguments in the proof of Theorem~\ref{T_weak_existence} at the end of Section~\ref{S_pf_Tex1}. The claimed convergence follows from Theorem~\ref{T_stability_MGP}.
\end{proof}

\begin{remark}
If $K$ is the Laplace transform of a $\R^{d\times d}$-valued measure~$\mu$,
\[
K(t)=\int_{\R_+} e^{-\lambda t}\mu(d\lambda), \quad t>0,
\]
then $K$ can indeed be approximated by weighted sums of exponentials. Constructions of such weighted sums are given by \citet{AJEE:19a}.
\end{remark}

\appendix

\section{Auxiliary results}

We occasionally use the following version of Young's inequality on subintervals. It uses the convolution notation $(f*g)(t)=\int_0^t f(t-s)g(s)ds$.

\begin{lemma}\label{L_Young}
Fix $T\in\R_+$ and $p,q,r\in[1,\infty]$ with $p^{-1}+q^{-1}=r^{-1}+1$. For any matrix-valued measurable functions $f,g$ on $[0,T]$ of compatible size, one has the Young type inequality $\|f*g\|_{L^r(0,T)}\le\|f\|_{L^p(0,T)}\|g\|_{L^q(0,T)}$.
\end{lemma}

\begin{proof}
This follows from the Young inequality for convolutions on the whole real line applied to the functions $|f|\bm1_{[0,T]}$ and $|g|\bm1_{[0,T]}$ that equal $|f(t)|$ and $|g(t)|$ for $t\in[0,T]$ and zero elsewhere.
\end{proof}

For ease of reference, we give the following well-known Gronwall type lemma for convolution inequalities; see \citet[Lemma~9.8.2]{GLS:90} for the case of non-convolution kernels.

\begin{lemma}\label{L_Gronwall}
Let $T\in\R_+$ and suppose $f,g,k\in L^1(0,T)$. Assume $k$ has a nonpositive resolvent $r\le 0$. If $f\le g-k*f$, then $f\le g-r*g$.
\end{lemma}

\begin{proof}
Write $f+k*f=g-h$ for $h\ge0$. By the definition of resolvent, one then has $f=(g-h)-r*(g-h)\le g-r*g$.
\end{proof}

\begin{lemma}\label{L_SI_welldef}
{Let $p \in [2, \infty)$.} Consider a convolution kernel $K\in L^p_{\rm loc}$ and a characteristic triplet $(b,a,\nu)$ satisfying \eqref{eq_GB_1}. Let $X$ be a predictable process with trajectories in $L^p_{\rm loc}$, and let $Z$ be an It\^o semimartingale whose differential characteristics (with respect to some given truncation function $\chi$) are $b(X),a(X),\nu(X,d\zeta)$. Then for almost every $t\in\R_+$, the stochastic integral $\int_{[0,t)} K(t-s)dZ_s$ is well-defined.
\end{lemma}

\begin{proof}
Define $\kappa(x)=|b(x)| + |a(x)| + \int_{\R^k} (1\wedge|\zeta|^2) \nu(x,d\zeta)$ and set $\tau_n=\inf\{t\colon\int_0^t |X_s|^pds>n\}$. Due to the bound \eqref{eq_GB_1} and the definition of $\tau_n$, we have $\int_0^{T\wedge\tau_n}\kappa(X_s) ds\le c(T+n)$. Thus, for any $T\in\R_+$, Young's inequality, see Lemma~\ref{L_Young}, gives
\begin{align*}
\int_0^T &\left( \int_0^{t\wedge\tau_n} |K(t-s)|^2 \kappa(X_s) ds\right)^{p/2} dt\\
&\le  \left(\int_0^T |K(t)|^p dt\right) \left(\int_0^{T\wedge\tau_n} \kappa(X_s) dt\right)^{p/2} \\
&\le \left(\int_0^T |K(t)|^p dt\right) \left( c(T+n)\right)^{p/2}.
\end{align*}
The right-hand side is deterministic; call it $c_n$. Taking expectations and using Tonelli's theorem yields
\[
\int_0^T \E\left[ \left( \int_0^{t\wedge\tau_n} |K(t-s)|^2 \kappa(X_s) ds\right)^{p/2}\right] dt \le c_n.
\]
Therefore, for each $n$, there is a nullset $N_n\subset[0,T]$ such that the expectation is finite for all $t\in[0,T]\setminus N_n$. The union $N=\bigcup_n N_n$ is still a nullset, and for each $t\in[0,T]\setminus N$,
\[
\int_0^{t\wedge\tau_n} |K(t-s)|^2 \kappa(X_s) ds < \infty \text{ for all $n$, $\mathbb{P}$-a.s.}
\]
Since $X$ has trajectories in $L^p_{\rm loc}$, we have $\tau_n\to\infty$. We infer that, for each $t\in[0,T]\setminus N$, $\int_0^t |K(t-s)|^2 \kappa(X_s) ds < \infty$, $\mathbb{P}$-a.s. This implies that the random variable $\int_{[0,t)} K(t-s)dZ_s$ is well-defined.
\end{proof}


\bibliographystyle{plainnat}
\bibliography{bibl}
\end{document}